\documentclass[12pt, reqno]{amsart}
\usepackage{mathrsfs}
\usepackage{amssymb,amsthm,amsmath}
\usepackage[numbers,sort&compress]{natbib}
\usepackage{amssymb,amsmath}
\usepackage{amsfonts}
\usepackage{mathrsfs}
\usepackage{latexsym}
\usepackage{amssymb}
\usepackage{amsthm}
\usepackage{color}
\usepackage{pdfsync}
\usepackage{indentfirst}
\date{today}

\usepackage{hyperref}
\hypersetup{hypertex=true,
	colorlinks=true,
	linkcolor=blue,
	anchorcolor=g,
	citecolor=red}

\usepackage{amsmath}
\usepackage{amsthm}

\newtheorem{remark}{Remark}[section]

\newtheorem{theorem}{Theorem}[section]
\newtheorem{proposition}{Proposition}[section]
\newtheorem{lemma}{Lemma}[section]
\newtheorem{corollary}{Corollary}[section]

\newcommand{\beq}{\begin{equation}}
	\newcommand{\eeq}{\end{equation}}
\newcommand{\ben}{\begin{eqnarray}}
	\newcommand{\een}{\end{eqnarray}}
\newcommand{\beno}{\begin{eqnarray*}}
	\newcommand{\eeno}{\end{eqnarray*}}

\numberwithin{equation}{section}

\begin{document}
	\title[Suppression of blow-up via the 3-D Couette flow]{Suppression of blow-up in
		Patlak-Keller-Segel system coupled with linearized Navier-Stokes equations via the 3D Couette flow}
	\author{Shikun~Cui}
	\address[Shikun~Cui]{School of Mathematical Sciences, Dalian University of Technology, Dalian, 116024,  China; Department of Mathematics and Statistics, McMaster University, Hamilton, Ontario, L8S 4K1, Canada}
	\email{cskmath@163.com}
	\author{Lili~Wang}
	\address[Lili~Wang]{School of Mathematical Sciences, Dalian University of Technology, Dalian, 116024,  China}
	\email{wanglili\_@mail.dlut.edu.cn}
	\author{Wendong~Wang}
	\address[Wendong~Wang]{School of Mathematical Sciences, Dalian University of Technology, Dalian, 116024,  China}
	\email{wendong@dlut.edu.cn}
	\date{\today}
	\maketitle

	\begin{abstract} It is known that finite-time blow-up in the 3D Patlak-Keller-Segel system  may occur for arbitrarily small values of the initial mass. It's interesting  whether one can prevent the finite-time blow-up via the stabilizing effect of the moving fluid.
		Consider the three-dimensional  Patlak-Keller-Segel system coupled with the linearized Navier-Stokes equations
		near the  Couette flow $(\ Ay, 0, 0 \ )$
		in a finite channel $\mathbb{T}\times\mathbb{I}\times\mathbb{T}$ with $ \mathbb{T}=[0,2\pi) $ and $ \mathbb{I}=[-1,1] $,
		with 
		the non-slip boundary condition, 
		and we show that if the  shear flow is sufficiently strong (A is large enough), then the solutions to  Patlak-Keller-Segel-Navier-Stokes system are global in time as long as 
		the initial cell mass is sufficiently small (for example, $M<\frac49$) and $ A\left(\|u_{2,0}(0)\|_{L^{2}}+\|u_{3,0}(0)\|_{L^{2}} \right)\leq C_{0} $, which 
		seems to be the first result of considering the suppression effect of Couette flow in the 3D Patlak-Keller-Segel-Navier-Stokes model, and also the first time considering the non-slip boundary condition.

	\end{abstract}
	
	
	\section{Introduction}
	Considering the following three-dimensional parabolic-elliptic Patlak-Keller-Segel (PKS) system coupled with the
	Navier-Stokes equations in a finite channel $ \mathbb{T}\times\mathbb{I}\times\mathbb{T} $ with $ \mathbb{T}=[0,2\pi) $ and $ \mathbb{I}=[-1,1] $:
	\begin{equation}\label{ini}
		\left\{
		\begin{array}{lr}
			\partial_tn+v\cdot\nabla n=\triangle n-\nabla\cdot(n\nabla c), \\
			\triangle c+n-c=0, \\
			\partial_tv+ v\cdot\nabla v+\nabla P=\triangle v+n\nabla\phi,\quad\nabla\cdot v=0, \\
			(n,v)\big|_{t=0}=(n_{\rm in},v_{\rm in}),
		\end{array}
		\right.
	\end{equation}
	where $n$ represents the cell density, $c$ denotes the chemoattractant density, and $v$ denotes the velocity of fluid. In addition, $P$ is the pressure and $\phi$ is the potential function. 
	
	If $v=0$ and $\phi=0$,  the system (\ref{ini}) is reduced to  the classical 3D parabolic-elliptic Patlak-Keller-Segel system.
	If $n=0$ and $c=0$,  the system (\ref{ini}) becomes the 3D Navier-Stokes equations.
	The Patlak-Keller-Segel system is proposed as a macroscopic model for chemotactic
	cell migration,
	which was jointly developed by Patlak \cite{Patlak1}, Keller and Segel \cite{Keller1}. 
	This system has wide applications in the fields of biology, ecology, and medicine.  It is of significant importance in cancer research, simulating bacterial diffusion behavior, and tissue development, among others.

	In recent years, extensive mathematical
	efforts have been undertaken to detect unbounded solutions of the PKS system.
	As long as the dimension of space is higher than one, the solutions of the classical PKS system may blow up in finite time. 
	The 2D PKS model of parabolic-parabolic has a critical mass of $8\pi$, if the cell mass $M:=||n_{\rm in}||_{L^1}$ is less than $8\pi$,  the solutions of the system are global in time \cite{Calvez1},
	if the cell mass is greater than $8\pi$, the solutions will blow up in finite time \cite{Schweyer1}.
	Moreover, the 2D parabolic-elliptic Patlak-Keller-Segel system is globally well-posed if and only if the total mass $M\leq8\pi$ by Wei in \cite{wei11}.
	When the spatial dimension is higher than two, the solutions of the PKS system will blow up for any initial mass, meaning that
	no mass threshold for aggregation exists in that case (for example, see \cite{HM}, \cite{nagai1995}, \cite{Na}, \cite{nagai1996}, \cite{sw2019}, \cite{winkler1} ).
	Therefore, 
	{\it an interesting question is to consider whether the stabilizing effect of the moving fluid can suppress the finite-time blow-up?}
	
	Let us first recall some results of shear flows in 2D briefly.
	
	{\bf I. The PKS system in 2D.} For the  parabolic-elliptic PKS system, Kiselev-Xu suppressed the  blow-up by stationary relaxation
	enhancing flows or time-dependent Yao-Zlatos near-optimal mixing flows in $\mathbb{T}^d$ \cite{Kiselev1}.
	Bedrossian-He \cite{Bedro2} also studied the suppression of blow-up by shear flows in
	$\mathbb{T}^2$ for the 2D parabolic-elliptic case. He \cite{he0} investigated the suppression of blow-up by a large strictly monotone shear flow for the parabolic-parabolic PKS model in $\mathbb{T}\times\mathbb{R}$. 
	
	{\bf II. The PKS-NS system in 2D.} For the coupled PKS-NS system, Zeng-Zhang-Zi considered the 2D PKS-NS system near the Couette flow in $\mathbb{T}\times\mathbb{R}$, and they
	proved that if the Couette flow is sufficiently strong, the solution to the system stays globally regular \cite{zeng}.
	He considered the blow-up suppression for the parabolic-elliptic PKS-NS system in
	$\mathbb{T}\times\mathbb{R}$ with the coupling of buoyancy effects \cite{he05} for a class of small initial data.
	Li-Xiang-Xu studied the suppression of blow-up in PKS-NS system via the Poiseuille flow in $\mathbb{T}\times\mathbb{R}$,
	and they showed that if Poiseuille flow is sufficiently strong, the solution is global in \cite{Li0} by assuming the smallness of the initial vorticity. Cui-Wang considered Poiseuille flow with the boundary of PKS-NS system and obtained the solutions are global regular without any smallness condition \cite{cui1}.
	
	{\bf III. The PKS system in 3D.}
	For the 3D PKS system of parabolic-elliptic case, Bedrossian-He investigated the suppression of blow-up by shear flows in $\mathbb{T}^3$ and $\mathbb{T}\times\mathbb{R}^2$ by assuming the initial mass is less than $8\pi$ in \cite{Bedro2}. 
	Feng-Shi-Wang \cite{Feng1} used the planar helical flows as transport
	flow to research the advective Kuramoto-Sivashinsky and
	Keller-Segel equations, and they  proved that when the amplitude of the flow is large enough, the $L^2$ norm of the solution is uniformly bounded in time. 
	Shi-Wang \cite{wangweike2} considered the suppression effect of the flow $(y,y^2,0)$ in $\mathbb{T}^2\times\mathbb{R}$, and Deng-Shi-Wang \cite{wangweike1} proved the Couette flow with a sufficiently large amplitude prevents
	the blow-up of solutions in the whole space. Besides, for the stability effect of buoyancy,
	Hu-Kiselev-Yao considered the blow-up suppression for the Patlak-Keller-Segel system coupled with a fluid flow that obeys Darcy's law for incompressible porous media via buoyancy force \cite{Hu0}.
	Hu and Kiselev proved that when the coupling is large enough, the Keller-Segel equations coupled with Stokes-Boussinesq flow is globally well-posed \cite{Hu1}, see also the recent result by Hu \cite{Hu2023}.
	
	{\bf IV. The PKS-NS system in 3D.} 
	For the 3D coupled PKS-NS system, it is still unknown whether the blow-up does not happen provided that the amplitude of some shear flow is sufficiently large. Suppression of blow-up via non-parallel shear flows was obtained in \cite{cui2} by Cui-Wang-Wang.
	As mentioned by Zeng-Zhang-Zi ({Remark 1.3 in \cite{zeng}} ):{\it ``It is very interesting to investigate the corresponding problems with boundary effects taken into account. "} 
	Exploring the potential suppression of blow-up when considering the non-slip boundary condition or the 3D Couette flow seems to have not been taken into account in this issue yet.   	
	Our main goal is to investigate these issues in this paper. 
	
	Before expressing our main theorem, we first introduce a perturbation $u$ around the three-dimensional Couette flow $(\ Ay,0,0\ )$, which $u(t,x,y,z)=v(t,x,y,z)-(\ Ay,0,0\ )$ satisfying $ u|_{t=0}=u_{\rm in}=(u_{1,\rm in}, u_{2,\rm in}, u_{3,\rm in}) $.  Then we rewrite the linearized form of system (\ref{ini}) into
	\begin{equation}\label{ini1}
		\left\{
		\begin{array}{lr}
			\partial_tn+Ay\partial_x n+u\cdot\nabla n-\triangle n=-\nabla\cdot(n\nabla c), \\
			\triangle c+n-c=0, \\
			\partial_tu+Ay\partial_x u+\left(
			\begin{array}{c}
				Au_2 \\
				0 \\
				0 \\
			\end{array}
			\right)
			-\triangle u+\nabla P=\left(
			\begin{array}{c}
				n \\
				0 \\
				0 \\
			\end{array}
			\right), \\
			\nabla \cdot u=0,
		\end{array}
		\right.
	\end{equation}
	together with the  boundary conditions
	\begin{equation}\label{boundary condition}
		\begin{aligned}
			n(t,x,\pm 1,z)=0,\quad
			c(t,x,\pm 1,z)=0.
		\end{aligned}
	\end{equation}
	In addition,  $u$ is imposed  the non-slip boundary condition 
	\begin{equation}\label{boundary condition2}
		\quad u(t,x,\pm 1,z)=0.
	\end{equation}

	\begin{remark}
		For the three dimensional Navier-Stokes equations, 3D lift-up effect is an important factor leading to its instability. In blow-up suppression, the lift-up effect also brings us difficulties. Here we select the special potential function $\phi=x$, for simplicity.
	\end{remark}
	
	Inspired by \cite{Chen1}, we introduce the vorticity $\omega_2=\partial_zu_1-\partial_xu_3$ and $\triangle u_2$, satisfying
	$$\partial_t\omega_2+Ay\partial_x\omega_2+A\partial_zu_2-\triangle\omega_2=\partial_{z}n,$$
	and 
	$$\partial_t\triangle u_2+Ay\partial_x \triangle u_2
	-\triangle^2 u_2=-\partial_{y}\partial_{x}n.$$
	
	After the time rescaling $t\mapsto\frac{t}{A}$, we get
	\begin{equation}\label{ini2}
		\left\{
		\begin{array}{lr}
			\partial_tn+y\partial_x n-\frac{1}{A}\triangle n=-\frac{1}{A}\nabla\cdot(u n)-\frac{1}{A}\nabla\cdot(n\nabla c), \\
			\triangle c+n-c=0, \\
			\partial_t\omega_2+y\partial_x\omega_2-\frac{1}{A}\triangle\omega_2=
			-\partial_zu_2+\frac{1}{A}\partial_{z}n, \\
			\partial_t\triangle u_2+y\partial_x \triangle u_2
			-\frac{1}{A}\triangle(\triangle u_2) =-\frac{1}{A}\partial_{y}\partial_{x}n, \\
			\nabla \cdot u=0,	\\
		\end{array}
		\right.
	\end{equation}
	with the boundary condition
	\begin{equation}\label{ini2_1}
		\left\{
		\begin{array}{lr}
			n(t,x,\pm1,z)=0,\\
			c(t,x,\pm1,z)=0,
			\\\omega_{2}(t, x,\pm 1, z)=0,
			\\
			\partial_{y}u_{2}(t,x,\pm1,z)=u_{2}(t,x,\pm 1,z)=0		
			.
		\end{array}
		\right.
	\end{equation}
	
	The main results of this paper are as follows.
	\begin{theorem}\label{result}
		Assume that the initial data $n_{\rm in}\in H^2(\mathbb{T}\times\mathbb{I}\times\mathbb{T})$, $u_{\rm in}\in H^2(\mathbb{T}\times\mathbb{I}\times\mathbb{T})$ and the initial cell mass $ M=\|n_{\rm in}\|_{L^{1}} $ satisfies $ C_{*}^{3}M<1 $ , where $ C_{*} $ is the sharp Sobolev constant of 
		\begin{equation*}
			\left\{
			\begin{array}{lr}
				\|f\|_{L^3(\mathbb{I}\times\mathbb{T})}\leq C_{*} \|f\|_{L^1(\mathbb{I}\times\mathbb{T})}^{\frac13}\|\nabla f\|_{L^2(\mathbb{I}\times\mathbb{T})}^{\frac23},\\
				f(y,z)|_{y=\pm 1}=0,\quad \int_{\mathbb{T}}f(y,\cdot)dz=0.	
			\end{array}
			\right.
		\end{equation*}
		Then there exists a positive constant $D_{0}$ depending on $||n_{\rm in}||_{H^{2}(\mathbb{T}\times\mathbb{I}\times\mathbb{T})}$
		and $||u_{\rm in}||_{H^2(\mathbb{T}\times\mathbb{I}\times\mathbb{T})}$, such that if $A\geq D_{0}$,
		and 
		\begin{equation}\label{u small}
			A\left(\|u_{2,0}(0)\|_{L^{2}}+\|u_{3,0}(0)\|_{L^{2}} \right)\leq C_{0},
		\end{equation}
		where $ C_{0} $ is an absolute constant,
		the solution of (\ref{ini2})-(\ref{ini2_1}) is global in time.
	\end{theorem}
	
	The following theorem provides a specific upper bound for the initial cell mass $ M. $	
	\begin{theorem}\label{result 1}
		Assume that the initial data $n_{\rm in}\in H^2(\mathbb{T}\times\mathbb{I}\times\mathbb{T})$, $u_{\rm in}\in H^2(\mathbb{T}\times\mathbb{I}\times\mathbb{T})$ and the initial cell mass $ M=\|n_{\rm in}\|_{L^{1}}<\frac49 $.
		Then there exists a positive constant $D_{0}$ depending on $||n_{\rm in}||_{H^{2}(\mathbb{T}\times\mathbb{I}\times\mathbb{T})}$
		and $||u_{\rm in}||_{H^2(\mathbb{T}\times\mathbb{I}\times\mathbb{T})}$, such that if $A\geq D_{0}$,
		and 
		\begin{equation}\label{u small}
			A\left(\|u_{2,0}(0)\|_{L^{2}}+\|u_{3,0}(0)\|_{L^{2}} \right)\leq C_{0},
		\end{equation}
		where $ C_{0} $ is an absolute constant,
		the solution of (\ref{ini2})-(\ref{ini2_1}) is global in time.
	\end{theorem}
	
	\begin{remark}
		The above results seem to be the first results of considering the suppression effect of the 3D Couette flow in the Patlak-Keller-Segel-Navier-Stokes model, and  also the first time considering the non-slip boundary. The constant in Theorem \ref{result 1} should not be optimal, and the existence of an optimal threshold value is still unknown. When considering the domain $\mathbb{T}^3$,  Bedrossian-He investigated  the 3D PKS system of parabolic-elliptic case and proved  the initial mass of less than $8\pi$ ensured the global existence in \cite{Bedro2}, where they used the logarithmic Hardy-Littlewood-Sobolev inequality and Green function on $\mathbb{T}^2 $; see also the recent parabolic-parabolic case in \cite{he24} by the similar idea. It's still unknown whether the similar result holds for the 3D PKS-NS system in a bounded domain of $\mathbb{T}\times \mathbb{I }\times \mathbb{T}$.
	\end{remark}
	
	\begin{remark}
		For non-linearized  Navier-Stokes equations, it seems to be  difficult to study the blow-up suppression of the system due to the existence of nonlinear terms $u\cdot\nabla u$. 
		We will consider this case in a forthcoming paper. 
	\end{remark}

	Here are some notations used in this paper.
	
	\noindent\textbf{Notations}:
	\begin{itemize}
		
		\item Define the Fourier transform by
		\begin{equation}
			f(t,x,y,z)=\sum_{k_1,k_3\in\mathbb{Z}}f^{k_1,k_3}(t,y){\rm e}^{i(k_1x+k_3z)}, \nonumber
		\end{equation}
		where $f^{k_1,k_3}(t,y)=\frac{1}{|\mathbb{T}|^2}\int_{\mathbb{T}\times\mathbb{T}}{f}(t,x,y,z){\rm e}^{-i(k_1x+k_3z)}dxdz.$
		For simplicity, denote $ \eta=\left(k_{1}^{2}+k_{3}^{2} \right)^{\frac12}. $
		
		\item For a given function $f=f(t,x,y,z)$,   write its $x$-part zero mode and  $x$-part non-zero mode by
		$$P_0f=f_0=\frac{1}{|\mathbb{T}|}\int_{\mathbb{T}}f(t,x,y,z)dx,\ {\rm and}\ P_{\neq}f=f_{\neq}=f-f_0.$$
		Especially, we use $u_{k,0}$ and $u_{k,\neq}$ to represent the zero mode
		and non-zero mode of the velocity $u_k(k=1,2,3)$, respectively.
		Similarly, we use $\omega_{k,0}$ and $\omega_{k,\neq}$ to represent the zero mode
		and non-zero mode of the vorticity $\omega_k(k=1,2,3)$.
		\item For the zero mode $ f_{0} $ of a given function $ f=f(t,x,y,z), $ we represent its $ z$-part zero mode and $ z$-part non-zero mode by
		$$f_{(0,0)}=\frac{1}{|\mathbb{T}|}\int_{\mathbb{T}}f_{0}(t,y,z)dz,\ {\rm and}\ f_{(0,\neq)}=f_{0}-f_{(0,0)}.$$

		\item The norm of the $L^p$ space and the time-space norm  $\|f\|_{L^qL^p}$ are defined as
		$$\|f\|_{L^p(\mathbb{T}\times\mathbb{I}\times\mathbb{T})}=\left(\int_{\mathbb{T}\times\mathbb{I}\times\mathbb{T}}|f|^p dxdydz\right)^{\frac{1}{p}},$$
		and
		$$\|f\|_{L^qL^p}=\big\|  \|f\|_{L^p(\mathbb{T}\times\mathbb{I}\times\mathbb{T})}\ \big\|_{L^q(0,t)}.$$
		\item Denote by $ M $ the total mass $ \|n(t)\|_{L^{1}}. $ Clearly, integration by parts and divergence theorem yield that $$  M:=\|n(t)\|_{L^{1}}=\|n_{\rm in}\|_{L^{1}}.$$
		
		\item Throughout this paper, we denote by $ C $ a positive constant independent of $A$, $t$ and the initial data, and it may be different from line to line.
	\end{itemize}
	
	The rest part of this paper is organized as follows. In Section \ref{sec 2}, the key idea and the proof of \textbf{Theorem \ref{result}} are presented. Section \ref{sec 3} is devoted to providing a collection of elementary
	lemmas, which are essential for the proof of \textbf{Proposition \ref{prop 1}} and \textbf{Proposition \ref{prop 2}}.
	In Section \ref{sec 4}, we finish the proof of \textbf{Proposition \ref{prop 1}}.
	The proof of \textbf{Proposition \ref{prop 2}} is established in Section \ref{sec 6} by using a key proposition in Section \ref{sec 5}. In Section \ref{sec 7}, we give the proof of {\textbf{Theorem \ref{result 1}}}. Some useful conclusions are shown in the appendix.
	
	\section{Key ideas and proof of Theorem \ref{result}}\label{sec 2}
	Noting that the enhanced dissipation of fluid only affects the non-zero mode, and it is essential to separate the zero mode and the non-zero mode of system (\ref{ini2}). For given functions $ f $ and $ g, $ there hold
	\begin{equation}\label{zero mode}
		\left(fg\right)_{0}=f_{0}g_{0}+(f_{\neq}g_{\neq})_{0},
	\end{equation}
	and
	\begin{equation}\label{non-zero mode}
		(fg)_{\neq}=f_0g_{\neq}+f_{\neq}g_0+(f_{\neq}g_{\neq})_{\neq}.
	\end{equation}
	For simplicity, we do not provide specific expressions for the zero mode and the non-zero mode of $ n, c, u, \omega_2 $ and $ \triangle u_{2}. $
	\begin{remark}\label{remark 1}
		It should be noted that zero mode and non-zero mode can be controlled by their own functions. That is, for any $ 1\leq p\leq \infty, $ there hold
		\begin{equation*}
			\begin{aligned}
				\|f_{0}\|_{L^{p}(\mathbb{T}\times\mathbb{I}\times\mathbb{T})}\leq\|f\|_{L^{p}(\mathbb{T}\times\mathbb{I}\times\mathbb{T})} ,
			\end{aligned}
		\end{equation*}
		and
		\begin{equation*}
			\|f_{\neq}\|_{L^{p}(\mathbb{T}\times\mathbb{I}\times\mathbb{T})}\leq\|f\|_{L^{p}(\mathbb{T}\times\mathbb{I}\times\mathbb{T})}+\|f_{0}\|_{L^{p}(\mathbb{T}\times\mathbb{I}\times\mathbb{T})}\leq 2\|f\|_{L^{p}(\mathbb{T}\times\mathbb{I}\times\mathbb{T})}.
		\end{equation*}	
		%
	\end{remark}	
	
	\begin{remark}
		In comparison to the 2D case, the equations satisfied by the zero mode of n are more complex in the 3D case.
		There are extra items $ \frac{1}{A}\partial_{z}(n_{0}\partial_{z}c_{0}) $, $\frac{1}{A}\partial_y(u_{2,0}n_0)$ and $\frac{1}{A}\partial_z(u_{3,0}n_0)$ for the 3D case.
	\end{remark}
	
	As in \cite{Chen1}, note that for the linear equation $ \partial_{t}f-\frac{1}{A}\triangle f+y\partial_{x}f=g $ in a finite channel, there are different space-time estimates corresponding to the non-slip boundary condition and Navier-slip boundary condition. Inspired by \cite{Chen1}, we next consider the equations (\ref{ini2}) with boundary conditions (\ref{ini2_1}) in frequency space. Denote
	\begin{equation*}
		\widehat{\triangle}=\widehat{\triangle}^{k_{1},k_{3}}=\partial_{y}^{2}-k_{1}^{2}-k_{3}^{2}.
	\end{equation*}
	Taking Fourier transform for (\ref{ini2})-(\ref{ini2_1}) with respect to $ (x,z), $ we obtain
	\begin{equation}\label{ini3}
		\left\{
		\begin{array}{lr}
			\partial_tn^{k_{1},k_{3}}-\frac{1}{A}\left(\partial_{y}^{2}-\eta^{2} \right)n^{k_{1}, k_{3}}+ik_{1}yn^{k_{1},k_{3}}\\\quad\quad\quad\quad=-\frac{1}{A}\left(ik_{1}, \partial_{y}, ik_{3} \right)\cdot(un)^{k_{1},k_{3}}-\frac{1}{A}\left(ik_{1}, \partial_{y}, ik_{3} \right)\cdot(n\nabla c)^{k_{1},k_{3}}, \\
			\partial_t\omega_{2}^{k_{1},k_{3}}-\frac{1}{A}\left(\partial_{y}^{2}-\eta^{2} \right)\omega_{2}^{k_{1},k_{3}}+ik_{1}y\omega_{2}^{k_{1},k_{3}}=-ik_{3}u_{2}^{k_{1},k_{3}}+\frac{1}{A}ik_{3}n^{k_{1},k_{3}} , \\
			\partial_t\widehat{\triangle}u_{2}^{k_{1},k_{3}}-\frac{1}{A}\left(\partial_{y}^{2}-\eta^{2} \right)\widehat{\triangle}u_{2}^{k_{1},k_{3}}+ik_{1}y\widehat{\triangle}u_{2}^{k_{1},k_{3}}=-\frac{1}{A}ik_{1}\partial_{y}n^{k_{1},k_{3}} , \\
			ik_{1}u_{1}^{k_{1},k_{3}}+\partial_{y}u_{2}^{k_{1},k_{3}}+ik_{3}u_{3}^{k_{1},k_{3}}=0,	\\
			n^{k_{1},k_{3}}|_{y=\pm1}=0,
			\\\omega_{2}^{k_{1},k_{3}}|_{y=\pm1 }=0,\\
			\partial_{y}u_{2}^{k_{1},k_{3}}|_{y=\pm 1}=u_{2}^{k_{1},k_{3}}|_{y=\pm 1}=0,
		\end{array}
		\right.
	\end{equation}
	where $ \eta=\sqrt{k_{1}^{2}+k_{3}^{2}}. $ 
	
	We introduce the following norms:
	\ben\label{eq:x a k}
	\|f\|_{X_{a}^{k_{1},k_{3}}}^{2}&=&\eta|k_{1}|\|e^{aA^{-\frac13}t}(-\partial_{y}, i\eta)f\|_{L^{2}L^{2}}^{2}+A^{-1}\eta^{2}\|e^{aA^{-\frac13}t}(\partial_{y}^{2}-\eta^{2})f\|_{L^{2}L^{2}}^{2}\nonumber \\&&+A^{-\frac32}\|e^{aA^{-\frac13}t}\partial_{y}(\partial_{y}^{2}-\eta^{2})f\|_{L^{2}L^{2}}^{2}+\eta^{2}\|e^{aA^{-\frac13}t}(-\partial_{y}, i\eta)f\|_{L^{\infty}L^{2}}^{2}\nonumber\\&&+A^{-\frac12}\|e^{aA^{-\frac13}t}(\partial_{y}^{2}-\eta^{2})f\|_{L^{\infty}L^{2}}^{2},
	\een
	\ben\label{eq:y a k}	\|f\|_{Y_{a}^{k_{1},k_{3}}}^{2}&=&\|e^{aA^{-\frac13}t}f\|_{L^{\infty}L^{2}}^{2}+A^{-1}\|e^{aA^{-\frac13}t}\partial_{y}f\|_{L^{2}L^{2}}^{2}\nonumber\\
	&&+\left((A^{-1}k_{1}^{2})^{\frac13}+A^{-1}\eta^{2} \right)\|e^{aA^{-\frac13}t}f\|_{L^{2}L^{2}}^{2},
	\een
	and
	\ben\label{eq:x a y a}
	\|f\|_{X_{a}}^{2}=\sum_{k_{1}\neq 0,k_{3}\in\mathbb{Z}}\|\hat{f}(k_{1},k_{3})\|_{X_{a}^{k_{1},k_{3}}}^{2},\quad \|f\|_{Y_{a}}^{2}=\sum_{k_{1}\neq 0,k_{3}\in\mathbb{Z}}\|\hat{f}(k_{1},k_{3})\|_{Y_{a}^{k_{1},k_{3}}}^{2}.
	\een
	It follows that
	\begin{equation}\label{X_a}
		\begin{aligned}
			&\|e^{aA^{-\frac13}t}\partial_{x}\nabla f_{\neq}\|_{L^{2}L^{2}}^{2}+A^{-1}\|e^{aA^{-\frac13}t}(\partial_{x},\partial_{z})\triangle f_{\neq}\|_{L^{2}L^{2}}^{2}+A^{-\frac32}\|e^{aA^{-\frac13}t}\partial_{y}\triangle f_{\neq}\|_{L^{2}L^{2}}^{2}\\&+\|e^{aA^{-\frac13}t}(\partial_{x},\partial_{z})\nabla f_{\neq}\|_{L^{\infty}L^{2}}^{2}+A^{-\frac12}\|e^{aA^{-\frac13}t}\triangle f_{\neq}\|_{L^{\infty}L^{2}}^{2}\leq C\|f\|_{X_{a}}^{2},
		\end{aligned}
	\end{equation}
	and
	\begin{equation}\label{Y_a}
		\|e^{aA^{-\frac13}t}f_{\neq}\|_{L^{\infty}L^{2}}^{2}+\frac{1}{A}\|e^{aA^{-\frac13}t}\nabla f_{\neq}\|_{L^{2}L^{2}}^{2}+\frac{1}{A^{\frac13}}\|e^{aA^{-\frac13}t}f_{\neq}\|_{L^{2}L^{2}}^{2}\leq C\|f\|_{Y_{a}}^{2}.
	\end{equation}

	Moreover, we set 
	\begin{equation*}
		\begin{aligned}
			E(t)=\|\partial_{x}\omega_{2,\neq}\|_{Y_{a}}+\|n_{\neq}\|_{Y_{a}}+\|u_{2,\neq}\|_{X_{a}} ,
		\end{aligned}
	\end{equation*}
	with the initial norm  
	\begin{equation*}
		\begin{aligned}
			E_{\rm in}=\|(\partial_{x}\omega_{2,\rm in})_{\neq}\|_{L^{2}}+\|(\triangle u_{\rm in})_{\neq}\|_{L^{2}}+\|n_{\rm in,\neq}\|_{L^{2}}.
		\end{aligned}
	\end{equation*}
	Let's designate $T$ as the terminal point of the largest range $[0, T]$ such that the following
	hypothesis hold
	\begin{align}
		E(t)\leq 2E_0, \label{assumption_0} \\
		||n||_{L^{\infty}L^{\infty}}\leq 2E_{1}, \label{assumption_1}
	\end{align}
	for any $t\in[0, T]$, where $ E_{0} $ and $E_1$ will be determined during the calculation.
	
	The following propositions are key to obtaining the main results. Combining them with the local well-posedness of the system (\ref{ini2})-(\ref{ini2_1}), we can deduce the global existence of the solution.
	\begin{proposition}\label{prop 1}
		Assume that the initial date $n_{\rm in}\in H^2(\mathbb{T}\times\mathbb{I}\times\mathbb{T})$ and $u_{\rm in}\in H^2(\mathbb{T}\times\mathbb{I}\times\mathbb{T})$, under the conditions of (\ref{u small}), (\ref{assumption_0}) and (\ref{assumption_1}), there exist a positive constant $ E_{0} $ depending on $ E_{\rm in}, $ and a positive constant $ D_{1}$ depending on $E_{0}, E_{1}, M $ and $\|u_{\rm in,0}\|_{H^{1}} $, such that if $ A\geq D_{1}, $ there holds
		\begin{equation*}
			\begin{aligned}
				E(t)&\leq E_{0},
			\end{aligned}
		\end{equation*}
		for all $ t\in[0,T]. $
	\end{proposition}
	
	\begin{proposition}\label{prop 2}
		Assume that the initial date $n_{\rm in}\in H^2(\mathbb{T}\times\mathbb{I}\times\mathbb{T})$, $u_{\rm in}\in H^2(\mathbb{T}\times\mathbb{I}\times\mathbb{T})$ and $ C_{*}^{3}M<1 $, under the conditions of (\ref{u small}), (\ref{assumption_0}) and (\ref{assumption_1}), there exists a positive constant $ E_{1} $ depending on $ E_{0}, M, \|u_{\rm in}\|_{H^{1}} $ and $ \|n_{\rm in}\|_{L^{2}\cap L^{\infty}}, $ such that 
		\begin{equation*}
			\begin{aligned}
				\|n\|_{L^{\infty}L^{\infty}}&\leq E_{1},
			\end{aligned}
		\end{equation*}
		for all $ t\in[0,T]. $
	\end{proposition}
	\begin{proof}[Proof of Theorem \ref{result}]
		Taking $ D_{0}=\max\{D_{1}, D_{2}\} $ and combining {\textbf{Proposition \ref{prop 1}}} and {\textbf{Proposition \ref{prop 2}}}, we complete the proof.	
	\end{proof}
	
	\section{A Priori estimates of $c$ and zero mode of $u$}\label{sec 3}

	\subsection{Elliptic estimates}
	We estimate $c$ by elliptic energy method.
	\begin{lemma}\label{ellip_0}
		Let $c_0$ and $n_{0}$ be the zero mode of $c$ and $n$, respectively, satisfying
		$$-\triangle c_0+c_0=n_{0},\quad c_{0}|_{y=\pm 1}=0.$$
		Then there hold
		\begin{align}
			\|\triangle c_0(t)\|_{L^2}+\|\nabla c_0(t)\|_{L^2}
			\leq C\|n_{0}(t)\|_{L^2},\nonumber 
		\end{align}
		and
		$$\|\nabla c_0(t)\|_{L^4}\leq C\|n_{0}(t)\|_{L^2},$$
		for any $t\geq0$.
	\end{lemma}
	\begin{proof}
		The basic energy estimates yield
		\begin{equation}
			\begin{aligned}
				\|\triangle c_0(t)\|^2_{L^2}+\|\nabla c_0(t)\|^2_{L^2}+\|c_0(t)\|^2_{L^2}
				\leq C\|n_{0}(t)\|^2_{L^2},
				\nonumber
			\end{aligned}
		\end{equation}
		which indicates
		$$\|\triangle c_0(t)\|_{L^2}+\|\nabla c_0(t)\|_{L^2}\leq C\|n_{0}(t)\|_{L^2}. $$
		Furthermore, using the Gagliardo-Nirenberg inequality, we have	
		$$\|\nabla c_0(t)\|_{L^4}\leq
		C\|\triangle c_0(t)\|^{\frac{1}{2}}_{L^2}\|\nabla c_0(t)\|^{\frac{1}{2}}_{L^2}+C\|\nabla c_{0}(t)\|_{L^{2}}
		\leq C\|n_{0}(t)\|_{L^2},$$
		since $\|\triangle c_0(t)\|_{L^2}=\|\nabla^2 c_0(t)\|_{L^2}$ due to the boundary condition of $c_0$.
	\end{proof}

	\begin{lemma}\label{ellip_2}
		Let $c_{\neq}$ and $n_{\neq}$ be the non-zero mode of $c$ and $n$,
		respectively, satisfying
		$$-\triangle c_{\neq}+c_{\neq}=n_{\neq},\quad c_{\neq}|_{y=\pm 1}=0.$$
		Then there hold
		\begin{align}
			\|\triangle c_{\neq}(t&)\|_{L^2}
			+\|\nabla c_{\neq}(t)\|_{L^2}\leq C\|n_{\neq}(t)\|_{L^2},\nonumber
		\end{align}
		and
		\begin{equation}
			\|\nabla c_{\neq}(t)\|_{L^4}\leq C\|n_{\neq}(t)\|_{L^2},\nonumber
		\end{equation}
		for any $t\geq0$.
	\end{lemma}
	\begin{proof}
		By integrating by parts, we have
		\begin{equation}
			\begin{aligned}
				\|\triangle c_{\neq}(t)\|^2_{L^2}+\|\nabla c_{\neq}(t)\|^2_{L^2}
				+\|c_{\neq}(t)\|^2_{L^2}
				&\leq C\|n_{\neq}(t)\|^2_{L^2}.
				\nonumber
			\end{aligned}
		\end{equation}
		Using the Gagliardo-Nirenberg inequality, we obtain
		$$\|\nabla c_{\neq}(t)\|_{L^4}\leq
		C\| c_{\neq}(t)\|^{\frac{1}{8}}_{L^2}
		\|\triangle c_{\neq}(t)\|^{\frac{7}{8}}_{L^2}+C\|c_{\neq}(t)\|_{L^{2}}
		\leq C\|n_{\neq}(t)\|_{L^2}.$$
	\end{proof}

	\subsection{A Priori estimates for zero mode of $u$}

	Before starting, we first prove two embedding inequalities for non-zero modulus functions.
	\begin{lemma}\label{lemma_0}
		Let $f$ be a function such that $f_{\neq}\in H^1(\mathbb{T}\times\mathbb{I}\times \mathbb{T})$, there holds
		$$||f_{\neq}||_{L^2(\mathbb{T}\times\mathbb{I}\times \mathbb{T})}
		\leq C ||\partial_xf_{\neq}||_{L^2(\mathbb{T}\times\mathbb{I}\times \mathbb{T})}\leq C ||\nabla f_{\neq}||_{L^2(\mathbb{T}\times\mathbb{I}\times \mathbb{T})}.$$
	\end{lemma}
	\begin{proof} It follows from Poincar\'{e}'s inequality immediately and we omit it.
	\end{proof}

	\begin{lemma}[]\label{lemma_delta_u}
		Assume that $u_{\neq}\in H^{2}(\mathbb{T}\times\mathbb{I}\times\mathbb{T}),$  there hold
		$$\left\|\left(
		\begin{array}{c}
			\partial_x \\
			\partial_z \\
		\end{array}
		\right)u_{\neq}\right\|_{L^2}\leq C(\|\omega_{2,\neq}\|_{L^2}
		+\|\nabla u_{2,\neq}\|_{L^2}),$$
		and
		$$\left\|\left(
		\begin{array}{c}
			\partial_x^2 \\
			\partial_z^2 \\
		\end{array}
		\right)u_{\neq}
		\right\|_{L^2}\leq C\left(\left\|\left(
		\begin{array}{c}
			\partial_x \\
			\partial_z \\
		\end{array}
		\right)\omega_{2,\neq}\right\|_{L^2}+\|\triangle u_{2,\neq}\|_{L^2}\right).$$
	\end{lemma}
	
	\begin{proof}
		Using the div-curl formula (see, for example, \cite{Von Wahl} or \textbf{Lemma 5.4} in \cite{Chen0}) and $$\partial_x u_{1,\neq}+\partial_y u_{2,\neq}+\partial_z u_{3,\neq}=0,$$ we have
		\begin{equation}
			\begin{aligned}
				\left\|\left(
				\begin{array}{c c}
					\partial_xu_{1,\neq}&\partial_zu_{1,\neq} \\
					\partial_xu_{3,\neq}&\partial_zu_{3,\neq} \\
				\end{array}
				\right)\right\|_{L^2}&\leq C(\|\partial_zu_{1,\neq}-\partial_xu_{3,\neq}\|_{L^2}
				+\|\partial_xu_{1,\neq}+\partial_zu_{3,\neq}\|_{L^2}) \\
				&\leq C(\|\omega_{2,\neq}\|_{L^2}+\|\partial_yu_{2,\neq}\|_{L^2}),
			\end{aligned}
		\end{equation}	
		which yields
		$$\left\|\left(
		\begin{array}{c}
			\partial_x \\
			\partial_z \\
		\end{array}
		\right)u_{\neq}
		\right\|_{L^2}\leq C(\|\omega_{2,\neq}\|_{L^2}+\|\nabla u_{2,\neq}\|_{L^2}).$$
		Similarly, one can also obtain
		\begin{equation}
			\begin{aligned}
				\left\|\left(
				\begin{array}{c}
					\partial_x^2 \\
					\partial_z^2 \\
				\end{array}
				\right)u_{\neq}
				\right\|_{L^2}\leq C\left(\left\|\left(
				\begin{array}{c}
					\partial_x \\
					\partial_z \\
				\end{array}
				\right)\omega_{2,\neq}\right\|_{L^2}+\|\triangle u_{2,\neq}\|_{L^2}\right).	\nonumber
			\end{aligned}
		\end{equation}	
		The proof is complete.
	\end{proof}

	Next we aims to obtain the estimate of $ \|u_{0}\|_{L^{\infty}L^{4}} $ for the zero mode of $u$.  Before that, we first estimate $ \|u_{0}\|_{L^{\infty}L^{2}}, $	which is stated as follows.
	\begin{lemma}[Estimate of $\|u_{0}\|_{L^{\infty} L^2}$]\label{u_infty2_1}
		Under the assumption of (\ref{assumption_1}), if
		\begin{equation}\label{small u}
			A\left(||u_{2,0}(0)||_{L^2}+\|u_{3,0}(0)\|_{L^{2}} \right)\leq C,
		\end{equation}
		there hold
		\begin{equation}\label{u10 infty 2}
			\|u_{1,0}\|_{L^{\infty}L^2}\leq C(
			\|u_{1,0}(0)\|_{L^2}
			+E_{1}+M+1),
		\end{equation}
		and
		\begin{equation}\label{u20 infty 2}
			\|u_{2,0}\|_{L^{\infty}L^2}+\|u_{3,0}\|_{L^{\infty}L^{2}}\leq \|u_{2,0}(0)\|_{L^2}+\|u_{3,0}(0)\|_{L^{2}}.
		\end{equation}
	\end{lemma}
	
	\begin{proof}
		It follows from \eqref{ini1} that $ u_{0} $ satisfies
		\begin{equation}\nonumber
			\partial_tu_{0}+\left(
			\begin{array}{c}
				u_{2,0} \\
				0 \\
				0 \\
			\end{array}
			\right)-\frac{1}{A}\triangle u_{0}+\frac{1}{A}\nabla P_{0}=\frac{1}{A}\left(
			\begin{array}{c}
				n_{0} \\
				0 \\
				0 \\
			\end{array}
			\right),\quad u_{0}|_{y=\pm 1}=0,
		\end{equation}	
		thus we have
		\begin{equation}\label{u_zero}
			\left\{
			\begin{array}{lr}
				\partial_tu_{1,0}-\frac{1}{A}\triangle u_{1,0}+u_{2,0}=\frac{1}{A}n_{0}, \\
				\partial_tu_{2,0}-\frac{1}{A}\triangle u_{2,0}+\frac{1}{A}\partial_y P_0=0, \\
				\partial_tu_{3,0}-\frac{1}{A}\triangle u_{3,0}+\frac{1}{A}\partial_z P_0=0
			\end{array}
			\right.
		\end{equation} 
		with
		\begin{equation*}
			u_{1,0}|_{y=\pm 1}=0,\quad u_{2,0}|_{y=\pm1}=0,\quad u_{3,0}|_{y=\pm1 }=0.
		\end{equation*}
		Due to ${\rm div}~u=0$, there holds 
		\begin{equation}\label{div u=0}
			\partial_{y}u_{2,0}+\partial_{z}u_{3,0}=0.
		\end{equation}
		For the equations of $ (\ref{u_zero})_{2} $	and $ (\ref{u_zero})_{3}, $ using (\ref{div u=0}), the basic energy estimate yields that
		\begin{equation*}
			\begin{aligned}
				&\frac12\frac{d}{dt}\left(\|u_{2,0}\|_{L^{2}}^{2}+\|u_{3,0}\|_{L^{2}}^{2} \right)+\frac{1}{A}\left(\|\nabla u_{2,0}\|_{L^{2}}^{2}+\|\nabla u_{3,0}\|_{L^{2}}^{2} \right)\\=&\frac{1}{A}\int_{\mathbb{I}\times\mathbb{T}}P_{0}\left(\partial_{y}u_{2,0}+\partial_{z}u_{3,0} \right)dydz=0,
			\end{aligned}
		\end{equation*}
		which implies \eqref{u20 infty 2}.
	
	Next we deal with $ u_{1,0}. $ Multiplying 2$ u_{1,0} $ on $ (\ref{u_zero})_{1}$ and integrating with $ (y,z) $ over $ \mathbb{I}\times\mathbb{T}, $ we obtain
	$$\frac{d}{dt}||u_{1,{0}}||_{L^2}^{2}+\frac{2}{A}||\nabla u_{1,{0}}||_{L^2}^{2}
	\leq 2\frac{||n_{{0}}||_{L^2}+A||u_{2,0}||_{L^2}}{A}||u_{1,{0}}||_{L^2}.$$
	Using  Poincar$\rm \acute{e}$'s inequality 
	\begin{equation*}
		\begin{aligned}
			\|u_{1,{0}}\|_{L^2}^2\leq C\|\nabla u_{1,{0}}\|^2_{L^2},
		\end{aligned}
	\end{equation*}
	we have
	\begin{equation}\label{result_0_0}
		\frac{d}{dt}||u_{1,{0}}||^2_{L^2}
		\leq-\frac{2\|u_{1,{0}}\|_{L^{2}}}{CA}\Big(||u_{1,{0}}||_{L^2}-C(||n_{{0}}||_{L^2}+A||u_{2,0}||_{L^2})\Big).
	\end{equation}
	This implies that
	\begin{equation}\label{omega}
		\|u_{1,{0}}(t)\|_{L^2}\leq 2C(\|u_{1,{0}}(0)\|_{L^2}+||n_{0}||_{L^{\infty}L^2}+A||u_{2,{0}}||_{L^{\infty}L^2}),
	\end{equation}
	for any $t\geq0$.
	Otherwise, there must exist $t=\check{t}>0$, such that
	\begin{equation}
		\begin{aligned}
			\|u_{1,{0}}(\check{t})\|_{L^2}= 
			2C(\|u_{1,{0}}(0)\|_{L^2}+||n_{0}||_{L^{\infty}L^2}+A||u_{2,0}||_{L^{\infty}L^2}),
			\label{result_2}
		\end{aligned}
	\end{equation}
	and
	\begin{equation}
		\begin{aligned}
			\frac{d}{dt}\left( \|u_{1,{0}}({t})\|^2_{L^2}\right)|_{t=\check{t}}\geq0.
			\label{result_3}
		\end{aligned}
	\end{equation}
	According to (\ref{result_0_0}) and (\ref{result_2}), we have
	\begin{equation}
		\begin{aligned}
			&\frac{d}{dt}\left(\|u_{1,{0}}(t)\|^2_{L^2}\right)|_{t=\check{t}}\\\leq&-\frac{2\|u_{1,0}(\check{t})\|_{L^{2}}}{CA}\left(2C\|u_{1,0}(0)\|_{L^{2}}+C\|n_{0}\|_{L^{\infty}L^{2}}+CA\|u_{2,0}\|_{L^{\infty}L^{2}} \right)<0.
			\label{result_4}
		\end{aligned}
	\end{equation}
	A contradiction arises between (\ref{result_3}) and (\ref{result_4}). Thus (\ref{omega}) holds.
	
	Thus
	$$\|u_{1,{0}}\|_{L^{\infty}L^2}\leq C(\|u_{1,{0}}(0)\|_{L^2}+||n_{0}||_{L^{\infty}L^2}+A||u_{2,{0}}||_{L^{\infty}L^2}).$$		
	Using \eqref{small u}, (\ref{u20 infty 2}) and
	\begin{equation*}
		\|n_{0}\|_{L^{\infty}L^{2}}\leq \|n\|_{L^{\infty}L^{\infty}}^{\frac12}\|n\|_{L^{\infty}L^{1}}^{\frac12}\leq E_{1}^{\frac12}M^{\frac12},
	\end{equation*}
	we obtain
	\beno
	\|u_{1,0}\|_{L^{\infty}L^2}\leq C(\|u_{1,0}(0)\|_{L^2} 
	+E_{1}+M+1),	\label{u11_1}
	\eeno
	which is just \eqref{u10 infty 2}.	
	
	Hence the proof is complete.	
	
\end{proof}

Next, we give the estimate of $ \|u_{0}\|_{L^{\infty}L^{4}}. $
\begin{lemma}[Estimate of $||u_{0}||_{L^{\infty}L^{4}}$]\label{lemma_2_1}
	Under the conditions of {\textbf{Lemma \ref{u_infty2_1}}} and the assumption (\ref{assumption_1}), there hold 
	\begin{equation}\label{u10 infty infty}
		\|u_{1,0}\|_{L^{\infty}L^{4}}\leq C\left(\|u_{1,0}(0)\|_{L^{4}}+\|u_{1,0}(0)\|_{L^{2}}+E_{1}+M+1 \right),
	\end{equation}
	\begin{equation}\label{u20 infty infty}
		\|u_{2,0}\|_{L^{\infty}L^{4}}\leq C\left(\|u_{2,0}(0)\|_{H^{1}}+\|u_{3,0}(0)\|_{H^{1}}+1 \right),
	\end{equation}
	and 
	\begin{equation}\label{u30 infty infty}
		\|u_{3,0}\|_{L^{\infty}L^{4}}\leq C\left(\|u_{2,0}(0)\|_{H^{1}}+\|u_{3,0}(0)\|_{H^{1}}+1 \right).
	\end{equation}
\end{lemma}
\begin{proof}
	\underline{\textbf{Step 1: Estimate of $ \|u_{1,0}\|_{L^{\infty}L^{4}} $.}} Multiplying $(\ref{u_zero})_{1}$ by $ 4u_{1,0}^{3} $ and integrating the  equation over $ \mathbb{I}\times\mathbb{T}, $  we obtain
	\begin{equation}\label{u2_11}
		\begin{aligned}
			\frac{d}{dt}\|u_{1,0}^{2}\|_{L^{2}}^{2}+\frac{3}{A}\|\nabla u_{1,0}^{2}\|_{L^{2}}^{2}=&-4\int_{\mathbb{I}\times\mathbb{T}}u_{2,0}u_{1,0}^{3}dydz+\frac{4}{A}\int_{\mathbb{I}\times\mathbb{T}}n_{0}u_{1,0}^{3}dydz\\\leq&\frac{4}{A}\left(A\|u_{2,0}\|_{L^{2}}+\|n_{0}\|_{L^{2}} \right)\|u_{1,0}^{3}\|_{L^{2}}.
		\end{aligned}
	\end{equation}	
	Thanks to Gagliardo-Nirenberg inequality
	\begin{equation*}
		\begin{aligned}
			\|u_{1, 0}^{3}\|_{L^2}&=\|u_{1, 0}\|^{3}_{L^{6}} \leq  \left(\|u_{1,0}^{2}\|_{L^{4}}^{2} \right)^{\frac{2}{3}}\|u_{1,0}\|_{L^{2}}^{\frac{1}{3}} \leq C( \|\nabla u^2_{1,0}\|_{L^{2}}\|u^2_{1,0}\|_{L^{2}})^{\frac{2}{3}} \|u_{1,0}\|^{\frac{1}{3}}_{L^{2}},
		\end{aligned}
	\end{equation*}
	then (\ref{u2_11}) yields
	\begin{equation}\label{u2_11_2}
		\begin{aligned}
			&\frac{d}{dt}\|u_{1,0}^{2}\|^2_{L^2}
			+\frac{3}{A}\|\nabla u_{1,0}^{2}\|^2_{L^2}\\
			\leq&\frac{C}{A}\left(A\|u_{2,0}\|_{L^{2}}+\|n_{0}\|_{L^{2}} \right)\|\nabla u_{1,0}^{2}\|_{L^{2}}^{\frac{2}{3}}\|u_{1,0}^{2}\|_{L^{2}}^{\frac{2}{3}}\|u_{1,0}\|_{L^{2}}^{\frac{1}{3}}\\\leq&\frac{2}{A}\|\nabla u_{1,0}^{2}\|_{L^{2}}^{2}+\frac{C}{A}\left( A\|u_{2,0}\|_{L^{2}}+\|n_{0}\|_{L^{2}}\right)^{\frac{3}{2}}\|u_{1,0}^{2}\|_{L^{2}}\|u_{1,0}\|_{L^{2}}^{\frac{1}{2}}.
		\end{aligned}
	\end{equation}	
	By the assumption (\ref{assumption_1}), (\ref{small u}) and (\ref{u20 infty 2}), we find
	\begin{equation}\label{u20 n}
		\begin{aligned}
			&	A\|u_{2,0}\|_{L^{\infty}L^{2}}+\|n_{0}\|_{L^{\infty}L^{2}}\\\leq&A\left(\|u_{2,0}(0)\|_{L^{2}}+\|u_{3,0}(0)\|_{L^{2}} \right)+\|n\|_{L^{\infty}L^{\infty}}^{\frac12}\|n\|_{L^{\infty}L^{1}}^{\frac12}\leq C\left(1+E_{1}+M\right),
		\end{aligned}
	\end{equation}
	and using (\ref{u10 infty 2}), one deduces
	\begin{equation}\label{u n}
		\begin{aligned}
			&\left(A\|u_{2,0}\|_{L^{\infty}L^{2}}+\|n_{0}\|_{L^{\infty}L^{2}} \right)^{\frac{3}{2}}\|u_{1,0}\|_{L^{\infty}L^{2}}^{\frac{1}{2}}\\\leq&C\left[\left(A\|u_{2,0}\|_{L^{\infty}L^{2}}+\|n_{0}\|_{L^{\infty}L^{2}} \right)^{2}+\|u_{1,0}\|_{L^{\infty}L^{2}}^{2} \right]\\\leq& C\left(\|u_{1,0}(0)\|_{L^{2}}^{2}+E_{1}^{2}+M^{2}+1 \right):=CB,
		\end{aligned}
	\end{equation}
	where $ B=\|u_{1,0}(0)\|_{L^{2}}^{2}+E_{1}^{2}+M^{2}+1. $ Moreover, due to Gagliardo-Nirenberg inequality, there holds
	\begin{equation}\label{u2_11_4}
		\begin{aligned}
			\| u^2_{1,0}\|^2_{L^{2}}
			\leq C\|\nabla u^2_{1,0}\|_{L^{2}}\|u^2_{1,0}\|_{L^{1}}.	
		\end{aligned}
	\end{equation}
	It follows from (\ref{u2_11_2}), (\ref{u20 n}), (\ref{u n}) and (\ref{u2_11_4}) that
	\begin{equation*}
		\begin{aligned}
			\frac{d}{dt}\|u_{1,0}^{2}\|^2_{L^2}
			\leq& -\frac{\|u^2_{1,0}\|^4_{L^{2}}}{ AC\|u^2_{1,0}\|^2_{L^{1}} }		
			+\frac{ CB }{ A }\|u_{1,0}^{2}\|_{L^{2}}\\=&-\frac{\|u_{1,0}^{2}\|_{L^{2}}}{AC\|u_{1,0}^{2}\|_{L^{1}}^{2}}\left(\|u_{1,0}^{2}\|_{L^{2}}^{3}-C^{2}B\|u_{1,0}^{2}\|_{L^{1}}^{2} \right)
			.
		\end{aligned}
	\end{equation*}	
	Similar to the proof of (\ref{omega}), applying the proof by contradiction, one deduces
	\begin{equation}\label{u10 p}
		\begin{aligned}
			\|u_{1,0}^{2}\|_{L^{\infty}L^{2}}^{3}\leq C\left(\|u_{1,0}^{2}(0)\|_{L^{2}}^{3}+B\|u_{1,0}^{2}\|_{L^{\infty}L^{1}}^{2} \right).
		\end{aligned}
	\end{equation}
	Substituting (\ref{u10 infty 2}) into (\ref{u10 p})	and recalling the definition of $ B $ in (\ref{u n}), we get
	\begin{equation}\label{u10 infty 4}
		\|u_{1,0}\|_{L^{\infty}L^{4}}\leq C\left(\|u_{1,0}(0)\|_{L^{4}}+\|u_{1,0}(0)\|_{L^{2}}+E_{1}+M+1 \right).
	\end{equation}

	\underline{\textbf{Step 2: Estimates of $ \|u_{2,0}\|_{L^{\infty}L^{4}} $ and $ \|u_{3,0}\|_{L^{\infty}L^{4}}. $}}
	Firstly, let's estimate $ \|\nabla u_{2,0}\|_{L^{\infty}L^{2}} $ and $ \|\nabla u_{3,0}\|_{L^{\infty}L^{2}}. $
	Multiplying $ (\ref{u_zero})_{2} $ by $ \partial_{t}u_{2,0} $ and add $ (\ref{u_zero})_{3}  $ by $ \partial_{t}u_{3,0}, $ then integrating the resulting equation with $ (y,z) $ over $ \mathbb{I}\times\mathbb{T} $ and using (\ref{div u=0}), we obtain
	\begin{equation*}
		\begin{aligned}
			&\|\partial_{t}u_{2,0}\|_{L^{2}}^{2}+\|\partial_{t}u_{3,0}\|_{L^{2}}^{2}+\frac{1}{2A}\frac{d}{dt}\left(\|\nabla u_{2,0}\|_{L^{2}}^{2}+\|\nabla u_{3,0}\|_{L^{2}}^{2}  \right)\\=&-\frac{1}{A}\int_{\mathbb{I}\times\mathbb{T}}\partial_{y}P_{0}\partial_{t}u_{2,0}dydz-\frac{1}{A}\int_{\mathbb{I}\times\mathbb{T}}\partial_{z}P_{0}\partial_{t}u_{3,0}dydz\\=&\frac{1}{A}\int_{\mathbb{I}\times\mathbb{T}}P_{0}\partial_{t}\left(\partial_{y}u_{2,0}+\partial_{z}u_{3,0} \right)dydz=0,
		\end{aligned}
	\end{equation*}
	which follows that
	\begin{equation}\label{u20' u30'}
		\begin{aligned}
			\|\nabla u_{2,0}\|_{L^{\infty}L^{2}}+\|\nabla u_{3,0}\|_{L^{\infty}L^{2}}\leq \|\nabla u_{2,0}(0)\|_{L^{2}}+\|\nabla u_{3,0}(0)\|_{L^{2}}.
		\end{aligned}
	\end{equation}		
	By interpolation inequality, (\ref{u20 infty 2}) and (\ref{u20' u30'}), we arrive at
	\begin{equation}\label{u20 infty 4}
		\begin{aligned}
			\|u_{2,0}\|_{L^{\infty}L^{4}}\leq& C\|u_{2,0}\|_{L^{\infty}L^{2}}^{\frac12}\|\nabla u_{2,0}\|_{L^{\infty}L^{2}}^{\frac12}\\\leq& C\left(\|u_{2,0}\|_{L^{\infty}L^{2}}+\|\nabla u_{2,0}\|_{L^{\infty}L^{2}} \right)\\\leq&C\left(\|u_{2,0}(0)\|_{H^{1}}+\|u_{3,0}(0)\|_{H^{1}}+1 \right),
		\end{aligned}
	\end{equation}
	and
	\begin{equation}\label{u30 infty 4}
		\|u_{3,0}\|_{L^{\infty}L^{4}}\leq C\left(\|u_{2,0}(0)\|_{H^{1}}+\|u_{3,0}(0)\|_{H^{1}}+1 \right).
	\end{equation}

	Collecting (\ref{u10 infty 4}), (\ref{u20 infty 4}) and (\ref{u30 infty 4}), we complete the proof.
\end{proof}

\begin{corollary}\label{corollary_1}
	It follows from \textbf{Lemma \ref{lemma_2_1}} that
	\begin{equation}\label{u0 infty infty}
		\begin{aligned}
			||u_0||_{L^{\infty}L^{4}}\leq C
			\left(\|u_{\rm in,0}\|_{H^{1}}+E_{1}+M+1 \right):=H_1.	
		\end{aligned}
	\end{equation}
\end{corollary}

\section{The estimate of $ E(t) $ and  proof of Proposition \ref{prop 1}}\label{sec 4}
\begin{proof}[Proof of Proposition \ref{prop 1}]
	The estimate of $ E(t) $ is divided into three terms, and we deal with them, respectively.
	
	\noindent\textbf{\underline{Step I. Estimate $||n_{\neq}||_{Y_a}$}.}
	Applying \textbf{Proposition \ref{slip}}  to  $(\ref{ini3})_1$, 
	we get 
	\begin{equation*}
		\begin{aligned}
			&\|n^{k_{1},k_{3}}\|_{Y_{a}^{k_{1},k_{3}}}^{2}\\\leq&C\left(\|n_{\rm in}^{k_{1},k_{3}}\|_{L^{2}}^{2}+\frac{1}{A}\|e^{aA^{-\frac13}t}(u_{2}n)^{k_{1},k_{3}}\|_{L^{2}L^{2}}^{2}+\frac{1}{A}\|e^{aA^{-\frac13}t}(n\partial_{y} c)^{k_{1},k_{3}}\|_{L^{2}L^{2}}^{2} \right)\\&+\frac{C}{A^{2}}\min\{(A^{-1}\eta^{2})^{-1}, (A^{-1}k_{1}^{2})^{-\frac13} \}\|e^{aA^{-\frac13}t}\left[k_{1}(u_{1}n)^{k_{1},k_{3}}+k_{1}(n\partial_{x} c)^{k_{1},k_{3}} \right]\|_{L^{2}L^{2}}^{2}\\&+\frac{C}{A^{2}}\min\{(A^{-1}\eta^{2})^{-1}, (A^{-1}k_{1}^{2})^{-\frac13} \}\|e^{aA^{-\frac13}t}\left[k_{3}(u_{3}n)^{k_{1},k_{3}}+k_{3}(n\partial_{z} c)^{k_{1},k_{3}} \right]\|_{L^{2}L^{2}}^{2}\\\leq&C\left(\|n_{\rm in}^{k_{1},k_{3}}\|_{L^{2}}^{2}+\frac{1}{A}\|e^{aA^{-\frac13}t}(un)^{k_{1},k_{3}}\|_{L^{2}L^{2}}^{2}+\frac{1}{A}\|e^{aA^{-\frac13}t}(n\nabla c)^{k_{1},k_{3}}\|_{L^{2}L^{2}}^{2} \right),
		\end{aligned}
	\end{equation*}
	and it follows from \eqref{eq:x a y a} that
	\begin{equation}\label{c_temp_11}
		\begin{aligned}
			||n_{\neq}||_{Y_a}
			\leq
			C\Big(\|n_{\rm in,\neq}\|_{L^2}
			+\frac{1}{A^{\frac{1}{2}}}\|{\rm e}^{aA^{-\frac{1}{3}}t}(un)_{\neq}\|_{L^2L^2}
			+\frac{1}{A^{\frac{1}{2}}}\|{\rm e}^{aA^{-\frac{1}{3}}t}(n\nabla c)_{\neq}\|_{L^2L^2}
			\Big).		
		\end{aligned}
	\end{equation}
	According to (\ref{non-zero mode}), there holds
	\begin{equation}\label{un_1}
		\begin{aligned}
			&\|{\rm e}^{aA^{-\frac{1}{3}}t}(un)_{\neq}\|_{L^2L^2}\\ \leq &C\Big(
			\|{\rm e}^{aA^{-\frac{1}{3}}t}u_0n_{\neq}\|_{L^2L^2}
			+\|{\rm e}^{aA^{-\frac{1}{3}}t}u_{\neq}n_{0}\|_{L^2L^2}
			+\|{\rm e}^{aA^{-\frac{1}{3}}t}(u_{\neq} n_{\neq})_{\neq}\|_{L^2L^2}\Big).
		\end{aligned}
	\end{equation}
	Due to 
	\begin{equation*}
		\|n_{\neq}\|_{L^{4}}\leq C\|n_{\neq}\|_{L^{2}}^{\frac14}\|\nabla n_{\neq}\|_{L^{2}}^{\frac34},
	\end{equation*}
	we get
	\ben\label{n neq 24}
	\|n_{\neq}\|_{L^{2}L^{4}}\leq C\left(\int_{0}^{t}\|n_{\neq}\|_{L^{2}}^{2}ds \right)^{\frac18}\left(\int_{0}^{t}\|\nabla n_{\neq}\|_{L^{2}}^{2}ds \right)^{\frac38}.
	\een
	Using (\ref{Y_a}), (\ref{n neq 24}) and \textbf{Corollary \ref{corollary_1}}, we have
	\begin{equation}\label{un_2}
		\begin{aligned}
			\|e^{aA^{-\frac13}t}u_{0}n_{\neq}\|_{L^{2}L^{2}}\leq&\|u_{0}\|_{L^{\infty}L^{4}}\|e^{aA^{-\frac13}t}n_{\neq}\|_{L^{2}L^{4}}\\\leq&CH_{1}\|e^{aA^{-\frac13}t}n_{\neq}\|_{L^{2}L^{2}}^{\frac14}\|e^{aA^{-\frac13}t}\nabla n_{\neq}\|_{L^{2}L^{2}}^{\frac34}\\\leq&CH_{1}A^{\frac{5}{12}}\|n_{\neq}\|_{Y_{a}}.
		\end{aligned}
	\end{equation}
	By (\ref{X_a}), (\ref{Y_a}), \textbf{Lemma \ref{lemma_0}} and \textbf{Lemma \ref{lemma_delta_u}}, there holds
	\begin{equation}\label{u neq 22}
		\begin{aligned}
			\|e^{aA^{-\frac13}t}u_{\neq}\|_{L^{2}L^{2}}\leq& C\|e^{aA^{-\frac13}t}\partial_{x}u_{\neq}\|_{L^{2}L^{2}}\\\leq& C\left(\|e^{aA^{-\frac13}t}\omega_{2,\neq}\|_{L^{2}L^{2}}+\|e^{aA^{-\frac13}t}\nabla u_{2,\neq}\|_{L^{2}L^{2}} \right)\\\leq&C\left(\|e^{aA^{-\frac13}t}\partial_{x}\omega_{2,\neq}\|_{L^{2}L^{2}}+\|e^{aA^{-\frac13}t}\partial_{x}\nabla u_{2,\neq}\|_{L^{2}L^{2}} \right)\\\leq&CA^{\frac16}\|\partial_{x}\omega_{2,\neq}\|_{Y_{a}}+C\|u_{2,\neq}\|_{X_{a}}.
		\end{aligned}
	\end{equation} 
	Thanks to (\ref{assumption_1}) and (\ref{u neq 22}), one obtains
	\begin{equation}\label{un_3}
		\begin{aligned}
			\|{\rm e}^{aA^{-\frac{1}{3}}t}u_{\neq} n_{0}\|_{L^2L^2}
			&\leq \|n\|_{L^{\infty}L^{\infty}}\|{\rm e}^{aA^{-\frac{1}{3}}t}u_{\neq}\|_{L^2L^2}\\
			&\leq CE_1A^{\frac{1}{6}}\|\partial_{x}\omega_{2,\neq}\|_{Y_{a}}+CE_{1}\|u_{2,\neq}\|_{X_{a}},
		\end{aligned}
	\end{equation}
	and 
	\begin{equation}\label{un_4}
		\begin{aligned}
			\|{\rm e}^{aA^{-\frac{1}{3}}t}(u_{\neq} n_{\neq})_{\neq}\|_{L^2L^2}
			&\leq C\|n\|_{L^{\infty}L^{\infty}}\|{\rm e}^{aA^{-\frac{1}{3}}t}u_{\neq}\|_{L^2L^2}\\
			&\leq CE_1A^{\frac{1}{6}}\|\partial_{x}\omega_{2,\neq}\|_{Y_{a}}+CE_{1}\|u_{2,\neq}\|_{X_{a}}.
		\end{aligned}
	\end{equation}
	Combining (\ref{un_1}), (\ref{un_2}), (\ref{un_3}) and (\ref{un_4}), there holds
	\begin{equation}\label{un_end}
		\begin{aligned}
			&\|{\rm e}^{aA^{-\frac{1}{3}}t}(un)_{\neq}\|_{L^2L^2}
			\\\leq& C(E_1+H_1)A^{\frac{5}{12}}\left(||n_{\neq}||_{Y_a}+||\partial_{x}\omega_{2,\neq}||_{Y_a}\right)+CE_{1}\|u_{2,\neq}\|_{X_{a}}.
		\end{aligned}
	\end{equation}
	
	Similarly as (\ref{un_1}), we have
	\begin{equation}\label{nc_1}
		\begin{aligned}
			&\|{\rm e}^{aA^{-\frac{1}{3}}t}(n\nabla c)_{\neq}\|_{L^2L^2} \\
			\leq& C\Big(
			\|{\rm e}^{aA^{-\frac{1}{3}}t}n_0\nabla c_{\neq}\|_{L^2L^2}
			+\|{\rm e}^{aA^{-\frac{1}{3}}t}n_{\neq}\nabla c_{0}\|_{L^2L^2}
			+\|{\rm e}^{aA^{-\frac{1}{3}}t}(n_{\neq} \nabla c_{\neq})_{\neq}\|_{L^2L^2}\Big).
		\end{aligned}
	\end{equation}
	From \textbf{Lemma \ref{ellip_0}} and the assumption (\ref{assumption_1}), we note that
	\begin{equation}\label{c0' infty 4}
		\|\nabla c_{0}\|_{L^{\infty}L^{4}}\leq C\|n_{0}\|_{L^{\infty}L^{2}}\leq C\|n_{0}\|_{L^{\infty}L^{\infty}}^{\frac12}\|n_{0}\|_{L^{\infty}L^{1}}^{\frac12}\leq CE_{1}^{\frac12}M^{\frac12}.
	\end{equation}
	Then it follows from (\ref{Y_a}), (\ref{n neq 24}) and (\ref{c0' infty 4}) that
	\begin{equation}\label{nc_2}
		\begin{aligned}	
			\|{\rm e}^{aA^{-\frac{1}{3}}t}n_{\neq}\nabla c_{0}\|_{L^2L^2}
			&\leq 
			\|\nabla c_{0}\|_{L^{\infty}L^{4}}
			\|{\rm e}^{aA^{-\frac{1}{3}}t}n_{\neq}\|_{L^2L^4} \\
			&\leq CE_1^{\frac{1}{2}}M^{\frac{1}{2}}		
			\|{\rm e}^{aA^{-\frac{1}{3}}t}n_{\neq}\|_{L^2L^2}^{\frac14}\|e^{aA^{-\frac13}t}\nabla n_{\neq}\|_{L^{2}L^{2}}^{\frac34}\\&
			\leq CA^{\frac{5}{12}}\left(E_{1}+M \right)
			||n_{\neq}||_{Y_a}.
		\end{aligned}
	\end{equation}
	Using (\ref{Y_a}), (\ref{assumption_1}) and \textbf{Lemma \ref{ellip_2}}, there holds
	\begin{equation}\label{nc_3}
		\begin{aligned}
			\|{\rm e}^{aA^{-\frac{1}{3}}t}n_0\nabla c_{\neq}\|_{L^2L^2}
			\leq& ||n_0||_{L^{\infty}L^{\infty}}
			\|{\rm e}^{aA^{-\frac{1}{3}}t}\nabla c_{\neq}\|_{L^2L^2}
			\\\leq& C||n||_{L^{\infty}L^{\infty}}
			\|{\rm e}^{aA^{-\frac{1}{3}}t}n_{\neq}\|_{L^2L^2} 
			\leq CE_1A^{\frac{1}{6}}||n_{\neq}||_{Y_a},
		\end{aligned}
	\end{equation}
	and 
	\begin{equation}\label{nc_4}
		\begin{aligned}
			\|{\rm e}^{aA^{-\frac{1}{3}}t}(n_{\neq} \nabla c_{\neq})_{\neq}\|_{L^2L^2}
			&	\leq 
			C\|n_{\neq}\|_{L^{\infty}L^{\infty}}\|{\rm e}^{aA^{-\frac{1}{3}}t}\nabla c_{\neq}\|_{L^2L^2} \\
			&\leq CE_1\|{\rm e}^{aA^{-\frac{1}{3}}t}n_{\neq}\|_{L^2L^2}
			\leq CE_1A^{\frac{1}{6}}||n_{\neq}||_{Y_a},
		\end{aligned}
	\end{equation}
	where we use $$\|n_{0}\|_{L^{\infty}L^{\infty}}+\|n_{\neq}\|_{L^{\infty}L^{\infty}}
	\leq 3\|n\|_{L^{\infty}L^{\infty}}\leq CE_1.$$
	Combining (\ref{nc_1}), (\ref{nc_2}), (\ref{nc_3}) and (\ref{nc_4}), there holds
	\begin{equation}\label{nc_end}
		\begin{aligned}
			\|{\rm e}^{aA^{-\frac{1}{3}}t}(n\nabla c)_{\neq}\|_{L^2L^2}
			\leq C(E_1+M)A^{\frac{5}{12}}||n_{\neq}||_{Y_a}.
		\end{aligned}
	\end{equation}
	Substituting (\ref{un_end}) and (\ref{nc_end}) into (\ref{c_temp_11}), we obtain 
	\begin{equation}\label{n_temp_11}
		\begin{aligned}
			&||n_{\neq}||_{Y_a}
			\\\leq&
			C\Big(\|n_{\rm in,\neq}\|_{L^2}
			+\frac{E_1+H_1+M}{A^{\frac{1}{12}}}(||\partial_{x}\omega_{2,\neq}||_{Y_a}+||n_{\neq}||_{Y_a}+|| u_{2,\neq}||_{X_a})
			\Big)\\\leq&C\left(\|n_{\rm in,\neq} \|_{L^{2}}+\frac{E_{1}^{2}+E_{0}^{2}+H_{1}^{2}+M^{2}}{A^{\frac{1}{12}}} \right).		
		\end{aligned}
	\end{equation}

	\noindent\textbf{~\underline{Step II. Estimate $ \|u_{2,\neq}\|_{X_{a}} $.}}
	Applying \textbf{Proposition \ref{noslip}}  to  $(\ref{ini3})_3$ and noting that $ \partial_{y}u_{2}^{k_{1},k_{3}}=-\left(ik_{1}u_{1}^{k_{1},k_{3}}+ik_{3}u_{3}^{k_{1},k_{3}} \right), $ we arrive at
	\begin{equation}\label{u2 neq}
		\begin{aligned}
			\|u_{2}^{k_{1},k_{3}}\|_{X_{a}^{k_{1},k_{3}}}^{2}\leq& C\left(
			\|\widehat{\triangle}u_{2,\rm in}^{k_{1},k_{3}}\|_{L^{2}}^{2}+(k_{1}^{2}+k_{3}^{2})^{-1}\left\|\widehat{\triangle}\left(ik_{1}u_{1,\rm in}^{k_{1},k_{3}}+ik_{3}u_{3,\rm in}^{k_{1},k_{3}} \right)\right\|_{L^{2}}^{2}
			\right)\\&+\frac{C}{A}\|e^{aA^{-\frac13}t}k_{1}n^{k_{1},k_{3}}\|_{L^{2}L^{2}}^{2}\\\leq&C\left(\|\widehat{\triangle}u_{\rm in}^{k_{1},k_{3}}\|_{L^{2}}^{2}+\frac{1}{A}\|e^{aA^{-\frac13}t}k_{1}n^{k_{1},k_{3}}\|_{L^{2}L^{2}}^{2} \right).
		\end{aligned}
	\end{equation}
	Using \eqref{eq:x a y a}, (\ref{Y_a}) and (\ref{n_temp_11}), (\ref{u2 neq}) yields
	\begin{equation}\label{u 2 neq}
		\begin{aligned}
			\|u_{2,\neq}\|_{X_{a}}\leq&C\left(\|(\triangle u_{\rm in})_{\neq}\|_{L^{2}}+\frac{1}{A^{\frac12}}\|e^{aA^{-\frac13}t}\partial_{x}n_{\neq}\|_{L^{2}L^{2}} \right)\\\leq&C\left(\|(\triangle u_{\rm in})_{\neq}\|_{L^{2}}+\|n_{\neq}\|_{Y_{a}} \right)\\\leq&C\left(\|(\triangle u_{\rm in})_{\neq}\|_{L^{2}}+\|n_{\rm in,\neq}\|_{L^{2}}+\frac{E_{1}^{2}+E_{0}^{2}+H_{1}^{2}+M^{2}}{A^{\frac{1}{12}}} \right).
		\end{aligned}
	\end{equation}

	\noindent\textbf{\underline{Step III. Estimate $\|\partial_{x}\omega_{2,\neq}\|_{Y_{a}}$}.}
	Applying \textbf{Proposition \ref{slip}}  to  $(\ref{ini3})_2$, we have
	\begin{equation}\label{omega'}
		\begin{aligned}
			&\|\omega_{2}^{k_{1},k_{3}}\|_{Y_{a}^{k_{1},k_{3}}}^{2}\\\leq&C\left(\|\omega_{2,{\rm in}}^{k_{1},k_{3}}\|_{L^{2}}^{2}+\frac{1}{A^{2}}\min\{ (A^{-1}\eta^{2})^{-1}, (A^{-1}k_{1}^{2})^{-\frac13}
			\}\|e^{aA^{-\frac13}t}k_{3}n^{k_{1},k_{3}}\|_{L^{2}L^{2}}^{2} \right)\\+&C\left(k_{3}^{2}(|k_{1}|\eta)^{-1} \right)\|e^{aA^{-\frac13}t}\partial_{y}u_{2}^{k_{1},k_{3}}\|_{L^{2}L^{2}}^{2}+C\left(k_{3}^{2}\eta|k_{1}|^{-1} \right)\|e^{aA^{-\frac13}t}u_{2}^{k_{1},k_{3}}\|_{L^{2}L^{2}}^{2} .
		\end{aligned}
	\end{equation}
	Noting that
	\begin{equation*}
		\begin{aligned}
			&	\sum_{k_{1}\neq 0,k_{3}\in\mathbb{Z}}k_{1}^{2}\left[\left(k_{3}^{2}(|k_{1}|\eta)^{-1} \right)\|e^{aA^{-\frac13}t}\partial_{y}u_{2}^{k_{1},k_{3}}\|_{L^{2}L^{2}}^{2}+\left(k_{3}^{2}\eta|k_{1}|^{-1} \right)\|e^{aA^{-\frac13}t}u_{2}^{k_{1},k_{3}}\|_{L^{2}L^{2}}^{2} \right]\\\leq&\sum_{k_{1}\neq 0,k_{3}\in\mathbb{Z}}\eta|k_{1}|\|e^{aA^{-\frac13}t}(\partial_{y}, i\eta)u_{2}^{k_{1},k_{3}}\|_{L^{2}L^{2}}^{2}\\\leq&\sum_{k_{1}\neq 0,k_{3}\in\mathbb{Z}}\|u_{2}^{k_{1},k_{3}}\|_{X_{a}^{k_{1},k_{3}}}^{2}=\|u_{2,\neq}\|_{X_{a}}^{2},
		\end{aligned}
	\end{equation*}
	it follows from (\ref{omega'}) that
	\begin{equation}\label{omega' 1}
		\begin{aligned}
			&	\|\partial_{x}\omega_{2,\neq}\|_{Y_{a}}^{2}\\\leq&C\sum_{k_{1}\neq 0,k_{3}\in\mathbb{Z}}k_{1}^{2}\left(\|\omega_{2,{\rm in}}^{k_{1},k_{3}}\|_{L^{2}}^{2}+\frac{1}{A}\|e^{aA^{-\frac13}t}n^{k_{1},k_{3}}\|_{L^{2}L^{2}}^{2} \right)\\+&C\sum_{k_{1}\neq 0,k_{3}\in\mathbb{Z}}k_{1}^{2}\left[\left(k_{3}^{2}(|k_{1}|\eta)^{-1} \right)\|e^{aA^{-\frac13}t}\partial_{y}u_{2}^{k_{1},k_{3}}\|_{L^{2}L^{2}}^{2}+\left(k_{3}^{2}\eta|k_{1}|^{-1} \right)\|e^{aA^{-\frac13}t}u_{2}^{k_{1},k_{3}}\|_{L^{2}L^{2}}^{2} \right]\\\leq&C\left(\|(\partial_{x}\omega_{2,\rm in})_{\neq}\|_{L^{2}}^{2}+\frac{1}{A}\|e^{aA^{-\frac13}t}\partial_{x}n_{\neq}\|_{L^{2}L^{2}}^{2}+\|u_{2,\neq}\|_{X_{a}}^{2} \right)\\\leq&C\left(\|(\partial_{x}\omega_{2,\rm in})_{\neq}\|_{L^{2}}^{2}+\|n_{\neq}\|_{Y_{a}}^{2}+\|u_{2,\neq}\|_{X_{a}}^{2}   \right).
		\end{aligned}		
	\end{equation}

	Substituting (\ref{n_temp_11}) and (\ref{u 2 neq}) into (\ref{omega' 1}), we conclude that
	\begin{equation}\label{omega' reslut}
		\begin{aligned}
			&\|\partial_{x}\omega_{2,\neq}\|_{Y_{a}}\\\leq& C\left( \|(\partial_{x}\omega_{2,\rm in})_{\neq}\|_{L^{2}}+\|(\triangle u_{\rm in})_{\neq}\|_{L^{2}}+\|n_{\rm in,\neq}\|_{L^{2}}+\frac{E_{1}^{2}+E_{0}^{2}+H_{1}^{2}+M^{2}}{A^{\frac{1}{12}}}\right).	
		\end{aligned}
	\end{equation}

	To sum up, we conclude that
	\begin{equation}\label{E(t)}
		E(t)\leq C\Big(E_{\rm in}+\frac{E_1^2+E_0^2+M^2+H_{1}^{2}}{A^{\frac{1}{12}}}\Big),
	\end{equation}
	where
	\begin{equation*}
		\begin{aligned}
			E_{\rm in}=\|(\partial_{x}\omega_{2,\rm in})_{\neq}\|_{L^{2}}+\|(\triangle u_{\rm in})_{\neq}\|_{L^{2}}+\|n_{\rm in,\neq}\|_{L^{2}}.
		\end{aligned}
	\end{equation*}
	Let us denote $ D_{1}:=\left(E_{1}^{2}+E_{0}^{2}+M^{2}+H_{1}^{2} \right)^{12}. $ Thus if $ A\geq D_{1}, $ (\ref{E(t)}) implies that
	\begin{equation*}
		E(t)\leq C\left(E_{\rm in}+1 \right):=E_{0}.
	\end{equation*}
	
	We complete the proof.
	
\end{proof}

\section{The $L^2$ estimate of zero mode of the density}\label{sec 5}

\begin{proposition}\label{priori0}
	Under the assumptions (\ref{assumption_0})-(\ref{assumption_1}), $ C_{*}^{3}M<1 $ and
	\begin{equation}\label{u}
		A\left(\|u_{2,0}(0)\|_{L^{2}}+\|u_{3,0}(0)\|_{L^{2}} \right)\leq C,
	\end{equation}
	there exist a positive constant $ D_{2} $ depending on $ E_{1}$ and $ E_{0}, $ and a positive constant $ H_{2} $ depending on $ \|n_{\rm in,0}\|_{L^{2}}, \|u_{\rm in,0}\|_{H^{1}} $ and $ M $ such that if $ A\geq D_{2}, $ there holds
	\begin{equation}\label{t1}
		\begin{aligned}
			\|n_{0}\|_{L^{\infty}L^{2}}\leq H_{2}.
		\end{aligned}		
	\end{equation}
\end{proposition}	
\begin{proof}
	Divide $ n_{0} $ into $ z $-part zero mode and $ z $-part non-zero mode as follows:
	\begin{equation*}
		n_{0}=n_{(0,0)}+n_{(0,\neq)},
	\end{equation*}
	thus
	\begin{equation}\label{n00 n0 neq}
		\|n_{0}\|_{L^{\infty}L^{2}}\leq \|n_{(0,0)}\|_{L^{\infty}L^{2}}+\|n_{(0,\neq)}\|_{L^{\infty}L^{2}}.
	\end{equation}
	
	Next, we estimate $ \|n_{(0,0)}\|_{L^{\infty}L^{2}} $ and $ \|n_{(0,\neq)}\|_{L^{\infty}L^{2}}, $ respectively.
	
	\subsection{\underline{Estimate of $ \|n_{(0,0)}\|_{L^{\infty}L^{2}} $}}	
	Recall in \eqref{ini2} that $ n_{0} $ satisfies
	\begin{equation}\label{n0 eq}
		\begin{aligned}
			\partial_t n_{0}-\frac{1}{A}\triangle n_{0}=&-\frac{1}{A}\left[\nabla\cdot(n_{\neq}\nabla c_{\neq})_{0}+\partial_{y}(n_{0}\partial_{y}c_{0})+\partial_{z}(n_{0}\partial_{z}c_{0}) \right]\\&-\frac{1}{A}\left[\nabla\cdot(u_{\neq}n_{\neq})_{0}+\partial_{y}(u_{2,0}n_{0})+\partial_{z}(u_{3,0}n_{0}) \right],
		\end{aligned}
	\end{equation}
	and note that $ u_{2,(0,0)}=0 $ due to $ {\rm div}~u=0, $ then $ n_{(0,0)} $ follows:
	\begin{equation*}
		\begin{aligned}
			&\partial_t n_{(0,0)}-\frac{1}{A}\partial_{yy} n_{(0,0)}\\=&-\frac{1}{A}\left[\partial_{y}(n_{\neq}\nabla c_{\neq})_{(0,0)}+\partial_{y}\left(n_{(0,0)}\partial_{y}c_{(0,0)} \right)+\partial_{y}\left(\left(n_{(0,\neq)}\partial_{y}c_{(0,\neq)} \right)_{(0,0)}\right) \right]\\&-\frac{1}{A}\left[\partial_{y}(u_{\neq}n_{\neq})_{(0,0)}+\partial_{y}\left(\left(u_{2,(0,\neq)}n_{(0,\neq)} \right)_{(0,0)} \right) \right].
		\end{aligned}
	\end{equation*}
	Multiplying the above equation by $ 2n_{(0,0)} $ and integrating the resulting equation with $ y $ over $ \mathbb{I}, $ noting that $ n_{(0,0)}|_{y=\pm 1}=0 $ , we obtain 
	\begin{equation}\label{temp1}
		\begin{aligned}
			&\frac{d}{dt}\|n_{(0,0)}\|^2_{L^2}
			+\frac{2}{A}\|\partial_{y} n_{(0,0)}\|^2_{L^2}
			\\=& \frac{2}{A}\int_{\mathbb{I}}(n_{\neq}\nabla c_{\neq})_{(0,0)}\partial_{y}n_{(0,0)}dy+\frac{2}{A}\int_{\mathbb{I}}(u_{\neq}n_{\neq})_{(0,0)}\partial_{y} n_{(0,0)}dy\\&+\frac{2}{A}\int_{\mathbb{I}}n_{(0,0)}\partial_{y}c_{(0,0)}\partial_{y}n_{(0,0)}dy+\frac{2}{A}\int_{\mathbb{I}}\left(n_{(0,\neq)}\partial_{y}c_{(0,\neq)} \right)_{(0,0)}\partial_{y}n_{(0,0)}dy\\&+\frac{2}{A}\int_{\mathbb{I}}\left(u_{2,(0\neq)}n_{(0,\neq)} \right)_{(0,0)}\partial_{y}n_{(0,0)}dy\\\leq&\frac{1}{A}\|\partial_{y}n_{(0,0)}\|_{L^{2}}^{2}+\frac{C}{A}\left(\|n_{\neq}\nabla c_{\neq}\|_{L^{2}}^{2}+\|u_{\neq}n_{\neq}\|_{L^{2}}^{2} \right)+\frac{C}{A}\|n_{(0,0)}\partial_{y}c_{(0,0)}\|_{L^{2}}^{2}\\&+\frac{C}{A}\|\left(n_{(0,\neq)}\partial_{y}c_{(0,\neq)} \right)_{(0,0)}\|_{L^{2}}^{2}+\frac{C}{A}\|\left(u_{2,(0,\neq)}n_{(0,\neq)} \right)_{(0,0)}\|_{L^{2}}^{2}.
		\end{aligned}
	\end{equation}
	Due to \eqref{ini2}, we have
	\begin{equation*}
		-\triangle c_{(0,\neq)}+c_{(0,\neq)}=n_{(0,\neq)}, \quad c_{(0,\neq)}|_{y=\pm 1}=0,
	\end{equation*}
	and
	\begin{equation*}
		-\partial_{yy}c_{(0,0)}+c_{(0,0)}=n_{(0,0)},\quad c_{(0,0)}|_{y=\pm 1}=0,
	\end{equation*}
	and the elliptic estimate gives
	\begin{equation}\label{ellip c 0 neq}
		\|\triangle c_{(0,\neq)}(t)\|_{L^2}^2+\| c_{(0,\neq)}\|_{L^{2}}^{2}
		+\|\nabla c_{(0,\neq)}(t)\|_{L^2}^2\leq C\|n_{(0,\neq)}(t)\|^2_{L^2},
	\end{equation}
	\begin{equation}\label{ellip c 0 neq 1}
		\begin{aligned}
			\|\nabla c_{(0,\neq)}\|_{L^{2}}^{2}+\|c_{(0,\neq)}\|_{L^{2}}^{2}=&\int_{\mathbb{I}\times\mathbb{T}}n_{(0,\neq)}c_{(0,\neq)}dydz\leq\|n_{(0,\neq)}\|_{L^{1}}\|c_{(0,\neq)}\|_{L^{\infty}}\\\leq&C\|n_{(0,\neq)}\|_{L^{1}}\|c_{(0,\neq)}\|_{L^{2}}^{\frac12}\|\triangle c_{(0,\neq)}\|_{L^{2}}^{\frac12}\leq CM\|n_{(0,\neq)}\|_{L^{2}},
		\end{aligned}
	\end{equation}
	and
	\begin{equation}\label{ellip c 00}
		\|\partial_{yy}c_{(0,0)}\|_{L^{2}}^{2}+\|\partial_{y}c_{(0,0)}\|_{L^{2}}^{2}\leq C\|n_{(0,0)}\|_{L^{2}}^{2}.
	\end{equation}
	Using \textbf{Lemma \ref{lemmaa1}}, (\ref{ellip c 0 neq}) and (\ref{ellip c 0 neq 1}), we have
	\begin{equation}\label{n partial_y c}
		\begin{aligned}
			&\|(n_{(0,\neq)}\partial_{y}c_{(0,\neq)})_{(0,0)}\|^2_{L^2}\\\leq&
			C\|n_{(0,\neq)}\|^2_{L^2}\left(\|\partial_{y}c_{(0,\neq)}\|_{L^2}
			\|\partial_{yy}c_{(0,\neq)}\|_{L^2}+\|\partial_{y}c_{(0,\neq)}\|_{L^{2}}^{2} \right)
			\\\leq& CM^{\frac12}\|n_{(0,\neq)}\|_{L^{2}}^{\frac72},
		\end{aligned}
	\end{equation}
	and
	\begin{equation}\label{u2 n}
		\begin{aligned}
			\|(u_{2,(0,\neq)}n_{(0,\neq)})_{(0,0)}\|^2_{L^2}
			&\leq 
			C\|n_{(0,\neq)}\|^2_{L^2}
			\|u_{2,(0,\neq)}\|_{L^2}
			\|\partial_{y}u_{2,(0,\neq)}\|_{L^2}\\
			&\leq C\|n_{(0,\neq)}\|^2_{L^2}\|u_{\rm in,0}\|^2_{H^1},
		\end{aligned}
	\end{equation}
	where we use (\ref{u20 infty 2}) and (\ref{u20' u30'}). By (\ref{ellip c 00}) and interpolation inequality, we get
	\begin{equation}\label{c00 infty}
		\|\partial_{y}c_{(0,0)}\|_{L^{\infty}}\leq C\|\partial_{y}c_{(0,0)}\|_{L^{2}}^{\frac12}\|\partial_{yy}c_{(0,0)}\|_{L^{2}}^{\frac12}+C\|\partial_{y}c_{(0,0)}\|_{L^{2}}\leq C\|n_{(0,0)}\|_{L^{2}},
	\end{equation}
	thus 
	\begin{equation}\label{n c'}
		\|n_{(0,0)}\partial_{y}c_{(0,0)}\|^2_{L^2}\leq \|\partial_{y}c_{(0,0)}\|^2_{L^{\infty}}\|n_{(0,0)}\|^2_{L^2}\leq C
		\|n_{(0,0)}\|^4_{L^2}.
	\end{equation}
	It follows from (\ref{temp1}), (\ref{n partial_y c}), (\ref{u2 n}) and (\ref{n c'}) that
	\begin{equation}\label{n00}
		\begin{aligned}
			\frac{d}{dt}\|n_{(0,0)}\|_{L^{2}}^{2}+\frac{1}{A}\|\partial_{y}n_{(0,0)}\|_{L^{2}}^{2}\leq&\frac{C}{A}\left(\|n_{\neq}\nabla c_{\neq}\|_{L^{2}}^{2}+\|u_{\neq}n_{\neq}\|_{L^{2}}^{2} \right)+\frac{C}{A}\|n_{(0,0)}\|_{L^{2}}^{4}\\&+\frac{C}{A}M^{\frac12}\|n_{(0,\neq)}\|_{L^{2}}^{\frac72}+\frac{C}{A}\|n_{(0,\neq)}\|_{L^{2}}^{2}\|u_{\rm in,0}\|_{H^{1}}^{2}.
		\end{aligned}
	\end{equation}
	Using Nash inequality
	\begin{equation*}
		-\|\partial_{y} n_{(0,0)}\|_{L^{2}}^{2}\leq -\frac{\|n_{(0,0)}\|_{L^{2}}^{6}}{C\|n_{(0,0)}\|_{L^{1}}^{4}}\leq -\frac{\|n_{(0,0)}\|_{L^{2}}^{6}}{CM^{4}},
	\end{equation*}	
	(\ref{n00}) yields 
	\begin{equation}\label{temp3}
		\begin{aligned}
			\frac{d}{dt}\|n_{(0,0)}\|_{L^{2}}^{2}\leq&-\frac{\|n_{(0,0)}\|_{L^{2}}^{6}}{CAM^{4}}+\frac{C}{A}\left(\|n_{\neq}\nabla c_{\neq}\|_{L^{2}}^{2}+\|u_{\neq}n_{\neq}\|_{L^{2}}^{2} \right)\\&+\frac{C}{A}\|n_{(0,0)}\|_{L^{2}}^{4}+\frac{C}{A}M^{\frac12}\|n_{(0,\neq)}\|_{L^{2}}^{\frac72}+\frac{C}{A}\|n_{(0,\neq)}\|_{L^{2}}^{2}\|u_{\rm in,0}\|_{H^{1}}^{2}.
		\end{aligned}
	\end{equation}
	Denote
	\begin{equation*}
		G(t):=\frac{C}{A}\int_{0}^{t}\left(\|n_{\neq}\nabla c_{\neq}\|_{L^{2}}^{2}+\|u_{\neq}n_{\neq}\|_{L^{2}}^{2} \right)ds,\quad{\rm for}\quad{\rm all}\quad t\geq 0.
	\end{equation*}		
	Using \textbf{Lemma \ref{ellip_2}}, assumptions (\ref{assumption_0})-(\ref{assumption_1}) and (\ref{u neq 22}), we note that
	\begin{equation*}
		\begin{aligned}
			G(t)\leq&\frac{C}{A}\|n_{\neq}\|_{L^{\infty}L^{\infty}}^{2}\left(\|n_{\neq}\|_{L^{2}L^{2}}^{2}+\|u_{\neq}\|_{L^{2}L^{2}}^{2} \right)\\\leq&\frac{C}{A}\|n\|_{L^{\infty}L^{\infty}}^{2}\left(A^{\frac16}\|n_{\neq}\|_{Y_{a}}+A^{\frac16}\|\partial_{x}\omega_{2,\neq}\|_{Y_{a}}+\|u_{2,\neq}\|_{X_{a}} \right)^{2}\\\leq&\frac{C}{A^{\frac23}}E_{1}^{2}E_{0}^{2}.
		\end{aligned}
	\end{equation*}
	Letting $ D_{2}:=\left(E_{1}^{2}E_{0}^{2} \right)^{\frac32}, $ choose $ A\geq D_{2}, $ and we arrive at 
	\begin{equation}\label{G(t)}
		G(t)\leq C.
	\end{equation}
	Then we can rewrite (\ref{temp3}) into
	{\small
		\begin{equation*}
			\begin{aligned}
				&\frac{d}{dt}\left(\|n_{(0,0)}\|_{L^{2}}^{2}-G(t) \right)\\\leq&-\frac{\|n_{(0,0)}\|_{L^{2}}^{6}}{CAM^{4}}+\frac{C}{A}\|n_{(0,0)}\|_{L^{2}}^{4}+\frac{C}{A}M^{\frac12}\|n_{(0,\neq)}\|_{L^{2}}^{\frac72}+\frac{C}{A}\|n_{(0,\neq)}\|_{L^{2}}^{2}\|u_{\rm in,0}\|_{H^{1}}^{2}\\=&-\frac{1}{CAM^{4}}\bigg( \|n_{(0,0)}\|_{L^{2}}^{6}-C^{2}M^{4}\|n_{(0,0)}\|_{L^{2}}^{4}-C^{2}M^{\frac92}\|n_{(0,\neq)}\|_{L^{2}}^{\frac72}-C^{2}M^{4}\|n_{(0,\neq)}\|_{L^{2}}^{2}\|u_{\rm in,0}\|_{H^{1}}^{2} \bigg)\\\leq&-\frac{1}{CAM^{4}}\bigg( \frac13\|n_{(0,0)}\|_{L^{2}}^{6}-\frac23(C^{2}M^{4})^{3}-C^{2}M^{\frac92}\|n_{(0,\neq)}\|_{L^{2}}^{\frac72}-C^{2}M^{4}\|n_{(0,\neq)}\|_{L^{2}}^{2}\|u_{\rm in,0}\|_{H^{1}}^{2}\bigg),
			\end{aligned}
		\end{equation*}
	}
	and this implies that
	\begin{equation}\label{n00 result}
		\begin{aligned}
			&\|n_{(0,0)}\|_{L^{\infty}L^{2}}\\\leq& C\left(\|n_{(0,0)}(0)\|_{L^{2}}+M^{2}+M^{\frac34}\|n_{(0,\neq)}\|_{L^{\infty}L^{2}}^{\frac{7}{12}}+M^{\frac23}\|n_{(0,\neq)}\|_{L^{\infty}L^{2}}^{\frac13}\|u_{\rm in,0}\|_{H^{1}}^{\frac13} +1\right)
			\\\leq&C\left(\|n_{(0,0)}(0)\|_{L^{2}}+M^{\frac34}\|n_{(0,\neq)}\|_{L^{\infty}L^{2}}^{\frac{7}{12}}+\|u_{\rm in,0}\|_{H^{1}}^{\frac{14}{13}}+M^{2}+1 \right)
			,
		\end{aligned}
	\end{equation}
	which is similar to the proof of $ (\ref{omega}) $, and we omit it.
	\subsection
	{\underline{Estimate of $ \|n_{(0,\neq)}\|_{L^{\infty}L^{2}} $}}
	From (\ref{n0 eq}), we know $ n_{(0,\neq)} $ satisfies 	
	\begin{equation}\label{n 0neq}
		\begin{aligned}
			&\partial_t n_{(0,\neq)}-\frac{1}{A}\triangle n_{(0,\neq)}\\=&-\frac{1}{A}\left[\nabla\cdot(n_{\neq}\nabla c_{\neq})_{(0,\neq)}+\partial_{y}(n_{0}\partial_{y}c_{0})_{(0,\neq)}+\partial_{z}(n_{0}\partial_{z}c_{0})_{(0,\neq)} \right]\\&-\frac{1}{A}\left[\nabla\cdot(u_{\neq}n_{\neq})_{(0,\neq)}+\partial_{y}(u_{2,0}n_{0})_{(0,\neq)}+\partial_{z}(u_{3,0}n_{0})_{(0,\neq)} \right].
		\end{aligned}
	\end{equation}
	Note that
	\begin{equation*}
		\begin{aligned}
			(n_{0}\partial_{y}c_{0})_{(0,\neq)}
			&=n_{(0,0)}\partial_{y}c_{(0,\neq)}
			+n_{(0,\neq)}\partial_{y}c_{(0,0)}
			+(n_{(0,\neq)}\partial_{y}c_{(0,\neq)})_{(0,\neq)}\\
			&=n_{(0,0)}\partial_{y}c_{(0,\neq)}
			+n_{(0,\neq)}\partial_{y}c_{(0,0)}
			+n_{(0,\neq)}\partial_{y}c_{(0,\neq)}-(n_{(0,\neq)}\partial_{y}c_{(0,\neq)})_{(0,0)};
		\end{aligned}
	\end{equation*}
	\begin{equation*}
		\begin{aligned}
			(n_{0}\partial_{z}c_{0})_{(0,\neq)}
			&=n_{(0,0)}\partial_{z}c_{(0,\neq)}
			+n_{(0,\neq)}\partial_{z}c_{(0,0)}
			+(n_{(0,\neq)}\partial_{z}c_{(0,\neq)})_{(0,\neq)}\\
			&=n_{(0,0)}\partial_{z}c_{(0,\neq)}
			+n_{(0,\neq)}\partial_{z}c_{(0,\neq)}
			-(n_{(0,\neq)}\partial_{z}c_{(0,\neq)})_{(0,0)};
		\end{aligned}
	\end{equation*}
	\begin{equation*}
		\begin{aligned}
			(u_{2,0}n_{0})_{(0,\neq)}
			&=u_{2,(0,0)}n_{(0,\neq)}
			+u_{2,(0,\neq)}n_{(0,0)}
			+(u_{2,(0,\neq)}n_{(0,\neq)})_{(0,\neq)}\\
			&=u_{2,(0,\neq)}n_{(0,0)}
			+u_{2,(0,\neq)}n_{(0,\neq)}
			-(u_{2,(0,\neq)}n_{(0,\neq)})_{(0,0)};
		\end{aligned}
	\end{equation*}	
	and
	\begin{equation*}
		\begin{aligned}
			(u_{3,0}n_{0})_{(0,\neq)}
			&=u_{3,(0,0)}n_{(0,\neq)}
			+u_{3,(0,\neq)}n_{(0,0)}
			+(u_{3,(0,\neq)}n_{(0,\neq)})_{(0,\neq)}\\
			&=u_{3,(0,0)}n_{(0,\neq)}+
			u_{3,(0,\neq)}n_{(0,0)}
			+u_{3,(0,\neq)}n_{(0,\neq)}
			-(u_{3,(0,\neq)}n_{(0,\neq)})_{(0,0)},
		\end{aligned}
	\end{equation*}  
	where we use $u_{2,(0,0)}=0,$
	then (\ref{n 0neq}) implies that
	\begin{equation}\label{n 0 neq}
		\begin{aligned}
			&\partial_t n_{(0,\neq)}-\frac{1}{A}\triangle n_{(0,\neq)}
			\\=&-\frac{1}{A}\left[\nabla\cdot(n_{\neq}\nabla c_{\neq})_{(0,\neq)}
			+\nabla\cdot(u_{\neq}n_{\neq})_{(0,\neq)}\right]\\
			&-\frac{1}{A}\left[\partial_y\big(n_{(0,0)}\partial_{y}c_{(0,\neq)}
			+n_{(0,\neq)}\partial_{y}c_{(0,0)}
			+n_{(0,\neq)}\partial_{y}c_{(0,\neq)}
			-(n_{(0,\neq)}\partial_{y}c_{(0,\neq)})_{(0,0)}\big)\right]\\
			&-\frac{1}{A}\left[\partial_z\big(n_{(0,0)}\partial_{z}c_{(0,\neq)}  
			+n_{(0,\neq)}\partial_{z}c_{(0,\neq)}\big)\right]\\
			&-\frac{1}{A}\left[\partial_y\big(u_{2,(0,\neq)}n_{(0,0)}
			+u_{2,(0,\neq)}n_{(0,\neq)}
			-(u_{2,(0,\neq)}n_{(0,\neq)})_{(0,0)}\big)\right]\\
			&-\frac{1}{A}\left[\partial_z\big(u_{3,(0,0)}n_{(0,\neq)}
			+u_{3,(0,\neq)}n_{(0,0)}
			+u_{3,(0,\neq)}n_{(0,\neq)}\big)\right].
		\end{aligned}
	\end{equation}
	Multiplying (\ref{n 0 neq}) by $ 2n_{(0,\neq)} $ and integrating it over $ (y,z)\in\mathbb{I}\times\mathbb{T}, $ noting that $ n_{(0,\neq)}|_{y=\pm 1}=0, $ one obtain
	\begin{equation}\label{n0 neq}
		\begin{aligned}
			\frac{d}{dt}\|n_{(0,\neq)}\|_{L^{2}}^{2}+\frac{2}{A}\|\nabla n_{(0,\neq)}\|_{L^{2}}^{2}= T_{1}(t)+T_{2}(t),
		\end{aligned}
	\end{equation}
	where
	\begin{equation*}
		T_{1}(t)=-\frac{2}{A}\int_{\mathbb{I}\times\mathbb{T}}\left[\partial_{y}(n_{(0,\neq)}\partial_{y}c_{(0,\neq)})+\partial_{z}(n_{(0,\neq)}\partial_{z}c_{(0,\neq)}) \right]n_{(0,\neq)}dydz,
	\end{equation*}
	and
	\begin{equation}\label{T_2}
		\begin{aligned}
			T_{2}(t)= &-\frac{2}{A}\int_{\mathbb{I}\times\mathbb{T}}\left[\nabla\cdot(n_{\neq}\nabla c_{\neq})_{(0,\neq)}
			+\nabla\cdot(u_{\neq}n_{\neq})_{(0,\neq)} \right]n_{(0,\neq)}dydz
			\\&-\frac{2}{A}\int_{\mathbb{I}\times\mathbb{T}}\left[\partial_{y}\left(n_{(0,0)}\partial_{y}c_{(0,\neq)} \right)+\partial_{z}\left(n_{(0,0)}\partial_{z}c_{(0,\neq)} \right) \right]n_{(0,\neq)}dydz\\&-\frac{2}{A}\int_{\mathbb{I}\times\mathbb{T}}\partial_{y}\left(n_{(0,\neq)}\partial_{y}c_{(0,0)} \right)n_{(0,\neq)}dydz\\&+\frac{2}{A}\int_{\mathbb{I}\times\mathbb{T}}\partial_{y}\left(n_{(0,\neq)}\partial_{y}c_{(0,\neq)} \right)_{(0,0)}n_{(0,\neq)}dydz\\&+\frac{2}{A}\int_{\mathbb{I}\times\mathbb{T}}\partial_{y}\left(u_{2,(0,\neq)}n_{(0,\neq)} \right)_{(0,0)}n_{(0,\neq)}dydz\\&-\frac{2}{A}\int_{\mathbb{I}\times\mathbb{T}}\left[\partial_{y}(u_{2,(0,\neq)}n_{(0,\neq)})+\partial_{z}(u_{3,(0,\neq)}n_{(0,\neq)}) \right]n_{(0,\neq)}dydz\\&-\frac{2}{A}\int_{\mathbb{I}\times\mathbb{T}}\left[\partial_{y}(u_{2,(0,\neq)}n_{(0,0)})+u_{3,(0,0)}\partial_{z}n_{(0,\neq)}+n_{(0,0)}\partial_{z}u_{3,(0,\neq)} \right]n_{(0,\neq)}dydz\\
			:=&T_{21}+\cdots+T_{27}.
		\end{aligned}
	\end{equation}
	For the term of $ T_{1}(t), $ recall $ -\triangle c_{(0,\neq)}=n_{(0,\neq)}-c_{(0,\neq)} $ in $ (\ref{ini2})_{2}$ and note that $ \|c_{(0,\neq)}\|_{L^{2}}^{\frac23}\|\nabla c_{(0,\neq)}\|_{L^{2}}^{\frac13}\leq C\|n_{(0,\neq)}\|_{L^{2}} $ from elliptic estimates, then by the interpolation inequalities
	\ben\label{eq:c*}
	\|n_{(0,\neq)}\|_{L^{3}}\leq C_{*}\|n_{(0,\neq)}\|_{L^{1}}^{\frac13}\|\nabla n_{(0,\neq)}\|_{L^{2}}^{\frac23},\quad \|c_{(0,\neq)}\|_{L^{3}}\leq C\|c_{(0,\neq)}\|_{L^{2}}^{\frac23}\|\nabla c_{(0,\neq)}\|_{L^{2}}^{\frac13},
	\een
	one deduces
	\begin{equation}\label{a7}
		\begin{aligned}
			T_{1}(t)=&\frac{1}{A}\int_{\mathbb{I}\times\mathbb{T}}\left(\partial_{y}c_{(0,\neq)}\partial_{y}n_{(0,\neq)}^{2}+\partial_{z}c_{(0,\neq)}\partial_{z}n_{(0,\neq)}^{2} \right)dydz\\=&-\frac{1}{A}\int_{\mathbb{I}\times\mathbb{T}}n_{(0,\neq)}^{2}\triangle c_{(0,\neq)}dydz\\=&\frac{1}{A}\int_{\mathbb{I}\times\mathbb{T}}n_{(0,\neq)}^{3}dydz-\frac{1}{A}\int_{\mathbb{I}\times\mathbb{T}}n_{(0,\neq)}^{2}c_{(0,\neq)}dydz\\\leq&\frac{C_{*}^{3}}{A}\|n_{(0,\neq)}\|_{L^{1}}\|\nabla n_{(0,\neq)}\|_{L^{2}}^{2}+\frac{1}{A}\|n_{(0,\neq)}\|_{L^{3}}^{2}\|c_{(0,\neq)}\|_{L^{3}}\\\leq&\frac{2C_{*}^{3}M}{A}\|\nabla n_{(0,\neq)}\|_{L^{2}}^{2}+\frac{C}{A}\|n_{(0,\neq)}\|_{L^{1}}^{\frac23}\|\nabla n_{(0,\neq)}\|_{L^{2}}^{\frac43}\|n_{(0,\neq)}\|_{L^{2}}\\\leq&\frac{2C_{*}^{3}M}{A}\|\nabla n_{(0,\neq)}\|_{L^{2}}^{2}+\frac{\delta}{A}\|\nabla n_{(0,\neq)}\|_{L^{2}}^{2}+\frac{C(\delta)}{A}M^{2}\|n_{(0,\neq)}\|_{L^{2}}^{3},
		\end{aligned}
	\end{equation}
	where $ \delta $ is a small positive constant.
	
	For the term of $ T_{2}(t), $  we  deal with the items at the right end of (\ref{T_2}), separately.
	
	{\bf $ \bullet $ Estimate of $ T_{21}$.}  Using integration by parts, H\"{o}lder's and Young inequalities, we obtain
	\begin{equation}\label{a0}
		\begin{aligned}
			T_{21}=&-\frac{2}{A}\int_{\mathbb{I}\times\mathbb{T}}\left[\nabla\cdot(n_{\neq}\nabla c_{\neq})_{(0,\neq)}
			+\nabla\cdot(u_{\neq}n_{\neq})_{(0,\neq)} \right]n_{(0,\neq)}dydz\\\leq&\frac{\delta}{A}\|\nabla n_{(0,\neq)}\|_{L^{2}}^{2}+\frac{C(\delta)}{A}\left(\|n_{\neq}\nabla c_{\neq}\|_{L^{2}}^{2}+\|u_{\neq}n_{\neq}\|_{L^{2}}^{2} \right).
		\end{aligned}
	\end{equation}

	{\bf $ \bullet $ Estimate of $ T_{22}$.}  By {\textbf{Lemma \ref{lemma_001}}} and elliptic estimate, we note that
	$$\|\nabla c_{(0,\neq)}\|_{L^\infty}\leq C\|\partial_z\triangle c_{(0,\neq)}\|^{\epsilon}_{L^2}\|\triangle c_{(0,\neq)}\|^{1-\epsilon}_{L^2}\leq C\|\partial_z n_{(0,\neq)}\|^{\epsilon}_{L^2}\|n_{(0,\neq)}\|^{1-\epsilon}_{L^2},$$
	where $\epsilon\in(0,1],$ there holds
	\begin{equation}\label{a1}
		\begin{aligned}
			T_{22}=&-\frac{2}{A}\int_{\mathbb{I}\times\mathbb{T}}\left[\partial_{y}\left(n_{(0,0)}\partial_{y}c_{(0,\neq)} \right)+\partial_{z}\left(n_{(0,0)}\partial_{z}c_{(0,\neq)} \right) \right]n_{(0,\neq)}dydz\\=&\frac{2}{A}\int_{\mathbb{I}\times\mathbb{T}}\left(n_{(0,0)}\partial_{y}c_{(0,\neq)}\partial_{y}n_{(0,\neq)}+n_{(0,0)}\partial_{z}c_{(0,\neq)}\partial_{z}n_{(0,\neq)} \right)dydz\\\leq&\frac{C}{A}\|\nabla c_{(0,\neq)}\|_{L^{\infty}}\|n_{(0,0)}\|_{L^{2}}\|\nabla n_{(0,\neq)}\|_{L^{2}}\\\leq&\frac{C}{A}\|\partial_z n_{(0,\neq)}\|^{\epsilon}_{L^2}\|n_{(0,\neq)}\|^{1-\epsilon}_{L^2}\|n_{(0,0)}\|_{L^2}
			\|\nabla n_{(0,\neq)}\|_{L^2}\\\leq&\frac{\delta}{A}\|\nabla n_{(0,\neq)}\|_{L^{2}}^{2}+\frac{C(\delta)}{A}\|n_{(0,\neq)}\|_{L^{2}}^{2}\|n_{(0,0)}\|_{L^{2}}^{\frac{2}{1-\epsilon}}.
		\end{aligned}
	\end{equation}
	
	{\bf $ \bullet $ Estimate of $ T_{23}$.}  Using (\ref{c00 infty}), one deduces
	\begin{equation}\label{a2}
		\begin{aligned}
			T_{23}=&	-\frac{2}{A}\int_{\mathbb{I}\times\mathbb{T}}\partial_{y}(n_{(0,\neq)}\partial_{y}c_{(0,0)} )n_{(0,\neq)}dydz\\=&\frac{2}{A}\int_{\mathbb{I}\times\mathbb{T}}n_{(0,\neq)}\partial_{y}c_{(0,0)}\partial_{y}n_{(0,\neq)}dydz\\\leq&\frac{2}{A}\|n_{(0,\neq)}\|_{L^{2}}\|\partial_{y}c_{(0,0)}\|_{L^{\infty}}\|\nabla n_{(0,\neq)}\|_{L^{2}}\\\leq&\frac{C}{A}\|n_{(0,\neq)}\|_{L^{2}}\|n_{(0,0)}\|_{L^{2}}\|\nabla n_{(0,\neq)}\|_{L^{2}}\\\leq&\frac{\delta}{A}\|\nabla n_{(0,\neq)}\|_{L^{2}}^{2}+\frac{C(\delta)}{A}\|n_{(0,\neq)}\|_{L^{2}}^{2}\|n_{(0,0)}\|_{L^{2}}^{2}.
		\end{aligned}
	\end{equation}
	
	{\bf $ \bullet $ Estimate of $ T_{24}$ and $ T_{25}$.}  From (\ref{n partial_y c}) and (\ref{u2 n}), we get
	\begin{equation}\label{a3}
		\begin{aligned}
			T_{24}=&	\frac{2}{A}\int_{\mathbb{I}\times\mathbb{T}}\partial_{y}\left(n_{(0,\neq)}\partial_{y}c_{(0,\neq)} \right)_{(0,0)}n_{(0,\neq)}dydz\\\leq&\frac{2}{A}\|\left(n_{(0,\neq)}\partial_{y}c_{(0,\neq)} \right)_{(0,0)}\|_{L^{2}}\|\nabla n_{(0,\neq)}\|_{L^{2}}\\\leq&\frac{C}{A}M^{\frac14}\|n_{(0,\neq)}\|_{L^{2}}^{\frac74}\|\nabla n_{(0,\neq)}\|_{L^{2}}\leq\frac{\delta}{A}\|\nabla n_{(0,\neq)}\|_{L^{2}}^{2}+\frac{C(\delta)}{A}M^{\frac12}\|n_{(0,\neq)}\|_{L^{2}}^{\frac72},
		\end{aligned}
	\end{equation}
	and
	\begin{equation}\label{a4}
		\begin{aligned}
			T_{25}=&\frac{2}{A}\int_{\mathbb{I}\times\mathbb{T}}\partial_{y}\left(u_{2,(0,\neq)}n_{(0,\neq)} \right)_{(0,0)}n_{(0,\neq)}dydz\\\leq&\frac{2}{A}\|\left(u_{2,(0,\neq)}n_{(0,\neq)} \right)_{(0,0)}\|_{L^{2}}\|\nabla n_{(0,\neq)}\|_{L^{2}}\\\leq&\frac{C}{A}\|u_{\rm in,0}\|_{H^{1}}\|n_{(0,\neq)}\|_{L^{2}}\|\nabla n_{(0,\neq)}\|_{L^{2}}\\\leq&\frac{\delta}{A}\|\nabla n_{(0,\neq)}\|_{L^{2}}^{2}+\frac{C(\delta)}{A}\|u_{\rm in,0}\|_{H^{1}}^{2}\|n_{(0,\neq)}\|_{L^{2}}^{2}.
		\end{aligned}
	\end{equation}
	
	{\bf $ \bullet $ Estimate of $ T_{26}$.}   Due to $ {\rm  div}~u_{(0,\neq)}=\partial_{y}u_{2,(0,\neq)}+\partial_{z}u_{3,(0,\neq)}=0,$  using integration by parts, there holds
	\begin{equation}\label{a5}
		\begin{aligned}
			T_{26}&-\frac{2}{A}\int_{\mathbb{I}\times\mathbb{T}}\left[\partial_{y}(u_{2,(0,\neq)}n_{(0,\neq)})+\partial_{z}(u_{3,(0,\neq)}n_{(0,\neq)} ) \right]n_{(0,\neq)}dydz\\=&-\frac{2}{A}\int_{\mathbb{I}\times\mathbb{T}}\left(u_{2,(0,\neq)}\partial_{y}n_{(0,\neq)}n_{(0,\neq)}+u_{3,(0,\neq)}\partial_{z}n_{(0,\neq)}n_{(0,\neq)} \right)dydz\\=&\frac{1}{A}\int_{\mathbb{I}\times\mathbb{T}}n_{(0,\neq)}^{2}\left(\partial_{y}u_{2,(0,\neq)}+\partial_{z}u_{3,(0,\neq)} \right)dydz=0.
		\end{aligned}
	\end{equation}
	
	{\bf $ \bullet $ Estimate of $ T_{27}$.}   Using (\ref{assumption_1}) and (\ref{u}), if $ A\geq E_{1}, $ we get
	\begin{equation}\label{a6}
		\begin{aligned}
			T_{27}=&-\frac{2}{A}\int_{\mathbb{I}\times\mathbb{T}}\left[\partial_{y}(u_{2,(0,\neq)}n_{(0,0)})+u_{3,(0,0)}\partial_{z}n_{(0,\neq)}+n_{(0,0)}\partial_{z}u_{3,(0,\neq)} \right]n_{(0,\neq)}dydz\\=&\frac{2}{A}\int_{\mathbb{I}\times\mathbb{T}}\left[u_{2,(0,\neq)}n_{(0,0)}\partial_{y}n_{(0,\neq)}-u_{3,(0,0)}\partial_{z}n_{(0,\neq)}n_{(0,\neq)}+n_{(0,0)}u_{3,(0,\neq)}\partial_{z}n_{(0,\neq)} \right]dydz\\\leq&\frac{C}{A}\|n\|_{L^{\infty}L^{\infty}}\left(\|u_{2,0}\|_{L^{\infty}L^{2}}+\| u_{3,0}\|_{L^{\infty}L^{2}} \right)\|\nabla n_{(0,\neq)}\|_{L^{2}}\\\leq&\frac{CE_{1}}{A^{2}}A\left(\|u_{2,0}(0)\|_{L^{2}}+\|u_{3,0}(0)\|_{L^{2}} \right)\|\nabla n_{(0,\neq)}\|_{L^{2}}\\\leq& \frac{C}{A}\|\nabla n_{(0,\neq)}\|_{L^{2}}\leq\frac{\delta}{A}\|\nabla n_{(0,\neq)}\|_{L^{2}}^{2}+\frac{C(\delta)}{A}.
		\end{aligned}
	\end{equation}
	
	Collecting (\ref{T_2}) and (\ref{a0})-(\ref{a6}), we arrive at
	\begin{equation}\label{T2}
		\begin{aligned}
			T_{2}(t)\leq&\frac{6\delta}{A}\|\nabla n_{(0,\neq)}\|_{L^{2}}^{2}+\frac{C(\delta)}{A}\left(\|n_{\neq}\nabla c_{\neq}\|_{L^{2}}^{2}+\|u_{\neq}n_{\neq}\|_{L^{2}}^{2} \right)\\&+\frac{C(\delta)}{A}\|n_{(0,\neq)}\|_{L^{2}}^{2}\|n_{(0,0)}\|_{L^{2}}^{\frac{2}{1-\epsilon}}+\frac{C(\delta)}{A}\|n_{(0,\neq)}\|_{L^{2}}^{2}\|n_{(0,0)}\|_{L^{2}}^{2}\\&+\frac{C(\delta)}{A}M^{\frac12}\|n_{(0,\neq)}\|_{L^{2}}^{\frac72}+\frac{C(\delta)}{A}\|u_{\rm in,0}\|_{H^{1}}^{2}\|n_{(0,\neq)}\|_{L^{2}}^{2}+\frac{C(\delta)}{A}.
		\end{aligned}
	\end{equation}
	Then it follows from (\ref{n0 neq}), (\ref{a7}) and (\ref{T2}) that
	\begin{equation}\label{n0 neq 1}
		\begin{aligned}
			&\frac{d}{dt}\|n_{(0,\neq)}\|_{L^{2}}^{2}+\frac{2}{A}\|\nabla n_{(0,\neq)}\|_{L^{2}}^{2}\\\leq&\frac{2C_{*}^{3}M}{A}\|\nabla n_{(0,\neq)}\|_{L^{2}}^{2}+\frac{7\delta}{A}\|\nabla n_{(0,\neq)}\|_{L^{2}}^{2}+\frac{C(\delta)}{A}\left(\|n_{\neq}\nabla c_{\neq}\|_{L^{2}}^{2}+\|u_{\neq}n_{\neq}\|_{L^{2}}^{2} \right)\\&+\frac{C(\delta)}{A}\|n_{(0,\neq)}\|_{L^{2}}^{2}\|n_{(0,0)}\|_{L^{2}}^{\frac{2}{1-\epsilon}}+\frac{C(\delta)}{A}\|n_{(0,\neq)}\|_{L^{2}}^{2}\|n_{(0,0)}\|_{L^{2}}^{2}\\&+\frac{C(\delta)}{A}M^{\frac12}\|n_{(0,\neq)}\|_{L^{2}}^{\frac72}+\frac{C(\delta)}{A}\|u_{\rm in,0}\|_{H^{1}}^{2}\|n_{(0,\neq)}\|_{L^{2}}^{2}\\&+\frac{C(\delta)}{A}+\frac{C(\delta)}{A}M^{2}\|n_{(0,\neq)}\|_{L^{2}}^{3}.
		\end{aligned}
	\end{equation}
	Noting that $ C_{*}^{3}M<1, $ then letting $ 7\delta=1-C_{*}^{3}M $ and substituting (\ref{n00 result}) into (\ref{n0 neq 1}), (\ref{n0 neq 1}) yields
	\begin{equation}\label{n0 neq 2}
		\begin{aligned}
			&\frac{d}{dt}\|n_{(0,\neq)}\|_{L^{2}}^{2}+\frac{7\delta}{A}\|\nabla n_{(0,\neq)}\|_{L^{2}}^{2}\\\leq&\frac{C}{A}\left(\|n_{\neq}\nabla c_{\neq}\|_{L^{2}}^{2}+\|u_{\neq}n_{\neq}\|_{L^{2}}^{2} \right)+\frac{C}{A}\|n_{(0,\neq)}\|_{L^{2}}^{2}\|n_{(0,0)}\|_{L^{2}}^{\frac{2}{1-\epsilon}}\\&+\frac{C}{A}\|n_{(0,\neq)}\|_{L^{2}}^{2}\|n_{(0,0)}\|_{L^{2}}^{2}+\frac{C}{A}M^{\frac12}\|n_{(0,\neq)}\|_{L^{2}}^{\frac72}+\frac{C}{A}\|u_{\rm in,0}\|_{H^{1}}^{2}\|n_{(0,\neq)}\|_{L^{2}}^{2}\\&+\frac{C}{A}+\frac{C}{A}M^{2}\|n_{(0,\neq)}\|_{L^{2}}^{3}\\\leq&\frac{C}{A}\left(\|n_{\neq}\nabla c_{\neq}\|_{L^{2}}^{2}+\|u_{\neq}n_{\neq}\|_{L^{2}}^{2} \right)+\frac{C}{A}M^{\frac{3}{2(1-\varepsilon)}}\|n_{(0,\neq)}\|_{L^{2}}^{\frac{19-12\varepsilon}{6(1-\varepsilon)}}\\&+\frac{C}{A}
			\left(1+\|n_{(0,0)}(0)\|_{L^{2}}^{\frac{2}{1-\varepsilon}}+\|u_{\rm in.0}\|_{H^{1}}^{\frac{28}{13(1-\varepsilon)}}+M^{\frac{4}{1-\varepsilon}} \right)\|n_{(0,\neq)}\|_{L^{2}}^{2}\\&+\frac{C}{A}M^{\frac12}\|n_{(0,\neq)}\|_{L^{2}}^{\frac72}+\frac{C}{A}+\frac{C}{A}M^{2}\|n_{(0,\neq)}\|_{L^{2}}^{3},
		\end{aligned}
	\end{equation}
	where we used $ \frac{19-12\epsilon}{6(1-\epsilon)}<4, $ i.e $ 0<\epsilon<\frac{5}{12}. $ Without losing generality, we choose $ \varepsilon=\frac{1}{6} $.  From (\ref{G(t)}), we find
	\begin{equation*}
		G(t)=\frac{C}{A}\int_{0}^{t}\left(\|n_{\neq}\nabla c_{\neq}\|_{L^{2}}^{2}+\|u_{\neq}n_{\neq}\|_{L^{2}}^{2} \right)ds\leq C,
	\end{equation*}
	providing $ A\geq \left(E_{1}^{2}E_{0}^{2} \right)^{\frac32}, $ and using Nash inequality
	\begin{equation*}
		-\|\nabla n_{(0,\neq)}\|_{L^{2}}^{2}\leq -\frac{\|n_{(0,\neq)}\|_{L^{2}}^{4}}{C\|n_{(0,\neq)}\|_{L^{1}}^{2}}\leq -\frac{\|n_{(0,\neq)}\|_{L^{2}}^{4}}{CM^{2}},
	\end{equation*}
	then we rewrite (\ref{n0 neq 2}) as follows:
	\beno
	&&\frac{d}{dt}\left(\|n_{(0,\neq)}\|_{L^{2}}^{2}-G(t) \right)\\&\leq&-\frac{7\delta}{CAM^{2}}\|n_{(0,\neq)}\|_{L^{2}}^{4} +\frac{C}{A}M^{\frac{9}{5}}\|n_{(0,\neq)}\|_{L^{2}}^{\frac{17}{5}}+\frac{C}{A}M^{\frac12}\|n_{(0,\neq)}\|_{L^{2}}^{\frac72}+\frac{C}{A}M^{2}\|n_{(0,\neq)}\|_{L^{2}}^{3}\\&&+\frac{C}{A}
	\left(1+\|n_{(0,0)}(0)\|_{L^{2}}^{\frac{12}{5}}+\|u_{\rm in.0}\|_{H^{1}}^{\frac{168}{65}}+M^{\frac{24}{5}} \right)\|n_{(0,\neq)}\|_{L^{2}}^{2}+\frac{C}{A}\\
	&=&-\frac{7\delta}{4CAM^{2}}\biggl[
	4\|n_{(0,\neq)}\|_{L^{2}}^{4}-\frac{4C^{2}M^{2+\frac95}}{7\delta}\|n_{(0,\neq)}\|_{L^{2}}^{\frac{17}{5}}-\frac{4C^{2}M^{2+\frac12}}{7\delta}\|n_{(0,\neq)}\|_{L^{2}}^{\frac72}	
	\\&&-\frac{4C^{2}M^{2}}{7\delta}\left(1+\|n_{(0,0)}(0)\|_{L^{2}}^{\frac{12}{5}}+\|u_{\rm in,0}\|_{H^{1}}^{\frac{168}{65}}+M^{\frac{24}{5}} \right)\|n_{(0,\neq)}\|_{L^{2}}^{2}\\&&-\frac{4C^{2}M^{2+2}}{7\delta}\|n_{(0,\neq)}\|_{L^{2}}^{3}-\frac{4C^{2}M^{2}}{7\delta}
	\biggr],
	\eeno
	which is controlled by
	\beno
	&&\-\frac{7\delta}{4CAM^{2}}\biggl\{ \frac{41}{40}\|n_{(0,\neq)}\|_{L^{2}}^{4}-\frac{3}{20}\left(\frac{4C^{2}M^{\frac{19}{5}}}{7\delta} \right)^{\frac{20}{3}}-\frac18\left(\frac{4C^{2}M^{\frac52}}{7\delta} \right)^{8}\\&&-\frac12\left[\frac{4C^{2}M^{2}}{7\delta}\left(1+\|n_{(0,0)}(0)\|_{L^{2}}^{\frac{12}{5}}+\|u_{\rm in,0}\|_{H^{1}}^{\frac{168}{65}}+M^{\frac{24}{5}} \right)\right]^{2}\\&&-\frac14\left(\frac{4C^{2}M^{4}}{7\delta} \right)^{4}-\frac{4C^{2}M^{2}}{7\delta}
	\biggr\}.
	\eeno
	Notice that $ 7\delta=1-C_{*}^{3}M  $, similar to the proof of $ (\ref{omega}), $ and the above inequality indicates that
	\begin{equation}\label{n0 neq result}
		\begin{aligned}
			&\|n_{(0,\neq)}\|_{L^{\infty}L^{2}}\\\leq& \frac{C}{\left(1-C_{*}^{3}M \right)^{2}}\left(\|n_{(0,\neq)}(0)\|_{L^{2}}+\|n_{(0,0)}(0)\|_{L^{2}}^{2}+\|u_{\rm in,0}\|_{H^{1}}^{2}+M^{7}+1 \right).
		\end{aligned}
	\end{equation}
	Recalling (\ref{n00 result}), there holds
	\begin{equation*}
		\|n_{(0,0)}\|_{L^{\infty}L^{2}}\leq C\left(\|n_{(0,0)}(0)\|_{L^{2}}+\|n_{(0,\neq)}\|_{L^{2}}+\|u_{\rm in,0}\|_{H^{1}}^{\frac{14}{13}}+M^{2}+1 \right).
	\end{equation*}
	Substituting  (\ref{n0 neq result}) into the above inequality, we obtain
	\begin{equation}\label{n00 result1}
		\begin{aligned}
			\|n_{(0,0)}\|_{L^{\infty}L^{2}}\leq \frac{C}{\left(1-C_{*}^{3}M \right)^{2}}\left(\|n_{\rm in,0}\|_{L^{2}}^{2}+\|u_{\rm in,0}\|_{H^{1}}^{2}+M^{7}+1 \right).
		\end{aligned}
	\end{equation}
	Then combining (\ref{n00 n0 neq}), (\ref{n0 neq result}) with (\ref{n00 result1}), we arrive at
	\begin{equation*}
		\begin{aligned}
			\|n_0\|_{L^{\infty}L^{2}}\leq \frac{C}{\left(1-C_{*}^{3}M \right)^{2}}\left(\|n_{\rm in,0}\|_{L^{2}}^{2}+\|u_{\rm in,0}\|_{H^{1}}^{2}+M^{7}+1 \right)=:H_{2}.
		\end{aligned}
	\end{equation*}

	The proof is complete.		
\end{proof}

\section{The $L^\infty $ estimate of the density and proof of Proposition \ref{prop 2}}\label{sec 6}
\begin{proof}[Proof of Propsition \ref{prop 2}]
	For $ p=2^{j} $ with $ j\geq 1, $ multiplying $ (\ref{ini2})_{1} $ by $2pn^{2p-1}$, and integrating by parts the resulting equation over $\mathbb{T}\times\mathbb{I}\times\mathbb{T}$, one deduces
	\begin{equation}
		\begin{aligned}
			\frac{d}{dt}||n^p||^2_{L^2}+\frac{2(2p-1)}{Ap}||\nabla n^{p}||_{L^2}^2
			&=\frac{2(2p-1)}{A}\int_{\mathbb{T}\times\mathbb{I}\times\mathbb{T}}n^{p}\nabla c\cdot\nabla n^{p}dxdydz \\
			&\leq\frac{2(2p-1)}{A}||n^p\nabla c||_{L^2}||\nabla n^p||_{L^2} \\
			&\leq\frac{2p-1}{Ap}||\nabla n^p||_{L^2}^2
			+\frac{(2p-1)p}{A}||n^p\nabla c||^2_{L^2}.\nonumber
		\end{aligned}
	\end{equation}
	Using H\"{o}lder  and Nash inequalities
	\begin{equation}
		\begin{aligned}
			||n^p\nabla c||_{L^2}^2
			\leq ||n^p||_{L^4}^2||\nabla c||_{L^4}^2
			\leq C||n^p||_{L^2}^{\frac{1}{2}}||\nabla n^p||_{L^2}^\frac{3}{2}||\nabla c||_{L^4}^2, \nonumber
		\end{aligned}
	\end{equation}
	we get
	\begin{equation}
		\begin{aligned}
			&\frac{d}{dt}||n^p||^2_{L^2}+\frac{2(2p-1)}{Ap}||\nabla n^{p}||_{L^2}^2 \\
			\leq&\frac{2p-1}{Ap}||\nabla n^p||_{L^2}^2
			+\frac{C(2p-1)p}{A}||n^p||_{L^2}^{\frac{1}{2}}||\nabla n^p||_{L^2}^\frac{3}{2}||\nabla c||_{L^4}^2 \\
			\leq&\frac{5(2p-1)}{4Ap}||\nabla n^p||_{L^2}^2
			+\frac{C(2p-1)p^7}{A}||n^p||_{L^2}^2||\nabla c||_{L^4}^8.\nonumber
		\end{aligned}
	\end{equation}
	Consequently
\begin{equation}
	\begin{aligned}
		\frac{d}{dt}||n^p||^2_{L^2}+\frac{1}{2A}||\nabla n^{p}||_{L^2}^2\leq
		\frac{Cp^8}{A}||n^p||_{L^2}^2||\nabla c||_{L^4}^8.
		\label{58}
	\end{aligned}
\end{equation}
Using the Nash inequality again\begin{equation}
	\begin{aligned}
		||n^p||_{L^2}
		\leq C||n^p||_{L^1}^{\frac{2}{5}}||\nabla n^p||_{L^2}^{\frac{3}{5}},
		\nonumber
	\end{aligned}
\end{equation}
we infer from (\ref{58}) that
\begin{equation}
	\begin{aligned}
		\frac{d}{dt}||n^p||^2_{L^2}
		\leq-\frac{||n^p||_{L^2}^{\frac{10}{3}}}{2AC||n^p||_{L^1}^{\frac{4}{3}}}
		+\frac{Cp^8}{A}||n^p||_{L^2}^2||\nabla c||_{L^{\infty}L^4}^8.
		\nonumber
	\end{aligned}
\end{equation}
Applying \textbf{Lemma \ref{ellip_0}}, \textbf{Lemma \ref{ellip_2}} and \textbf{Proposition  \ref{priori0}}, there holds
\begin{equation*}
	\begin{aligned}
		||\nabla c||_{L^{\infty}L^4}\leq& ||\nabla c_{\neq}||_{L^{\infty}L^4}+||\nabla c_0||_{L^{\infty}L^4}
		\\\leq& C\left(\|n_{\neq}\|_{L^{\infty}L^{2}}+\|n_{0}\|_{L^{\infty}L^{2}} \right) \leq C\left(E_{0}+H_{2} \right).
	\end{aligned}
\end{equation*}
Therefore
\begin{equation}\label{np}
	\begin{aligned}
		\frac{d}{dt}||n^p||^2_{L^2}
		\leq-\frac{||n^p||^{\frac{10}{3}}_{L^2}}{2CA||n^p||_{L^1}^{\frac{4}{3}}}
		+\frac{Cp^8}{A}||n^p||^2_{L^2}(E_{0}^8+H_{2}^8).
	\end{aligned}
\end{equation}

Claim that
\begin{equation}
	\begin{aligned}
		\sup_{t\geq 0}||n^p||_{L^2}^2\leq \max\Big\{8C^3(E_{0}^{8}+H_{2}^{8})^{\frac32}p^{12}\sup_{t\geq0}||n^p||^2_{L_1}, 2||n_{\rm in}^p||_{L^2}^2\Big\}\label{n1_1}.
	\end{aligned}
\end{equation}
Otherwise, there must exist $ t=\tilde{t}>0 $ such that
\begin{equation}\label{t}
	\|n^{p}(\tilde{t})\|_{L^{2}}^{2}=\max\{8C^{3}(E_{0}^{8}+H_{2}^{8})^{\frac32}p^{12}\|n^{p}(\tilde{t})\|_{L^{1}}^{2},2\|n_{\rm in}^{p}\|_{L^{2}}^{2} \},
\end{equation}
and
\begin{equation}\label{t 1}
	\frac{d}{dt}\left(\|n^{p}(t)\|_{L^{2}}^{2} \right)|_{t=\tilde{t}}\geq 0.
\end{equation}
According to (\ref{np}) and (\ref{n1_1}), we have
\begin{equation}\label{t 2}
	\begin{aligned}
		&\frac{d}{dt}\left(\|n^{p}(t)\|_{L^{2}}^{2} \right)|_{t=\tilde{t}}\\\leq&-\|n^{p}(\tilde{t})\|_{L^{2}}^{2}\left[\frac{\|n^{p}(\tilde{t})\|_{L^{2}}^{\frac43}}{2CA\|n^{p}(\tilde{t})\|_{L^{1}}^{\frac43}}+\frac{Cp^{8}}{A}(E_{0}^{8}+H_{2}^{8}) \right]\\\leq&-\|n^{p}(\tilde{t})\|_{L^{2}}^{2}\left[\frac{\left[8C^{3}(E_{0}^{8}+H_{2}^{8})^{\frac32}p^{12}\|n^{p}(\tilde{t})\|_{L^{1}}^{2} \right]^{\frac23}}{2CA\|n^{p}(\tilde{t})\|_{L^{1}}^{\frac43}}-\frac{Cp^{8}}{A}(E_{0}^{8}+H_{2}^{8}) \right]\\\leq&-\|n^{p}(\tilde{t})\|_{L^{2}}^{2}\frac{Cp^{8}(E_{0}^{8}+H_{2}^{8})}{A}<0,
	\end{aligned}
\end{equation}
which is a contradiction due to (\ref{t 1}) and (\ref{t 2}). Thus, (\ref{n1_1}) is proved.

Next, the Moser-Alikakos iteration is used to determine $E_{1}$. Recall $ p=2^{j} $ with $ j\geq 1, $
and we rewrite (\ref{n1_1}) into
\begin{equation}
	\begin{aligned}
		&\sup_{t\geq 0}\int_{\mathbb{T}\times\mathbb{I}\times\mathbb{T}}|n(t)|^{2^{j+1}}dxdydz \\
		\leq& \max\Big\{C_{1}\Big(\sup_{t\geq0}\int_{\mathbb{T}\times\mathbb{I}\times\mathbb{T}}|n(t)|^{2^{j}}dxdydz\Big)^2, 2\int_{\mathbb{T}\times\mathbb{I}\times\mathbb{T}}|n_{\rm in}|^{2^{j+1}}dxdydz\Big\},\label{n1_2}
	\end{aligned}
\end{equation}
where $C_1=8C^3(E_{0}^{8}+H_{2}^{8})^{\frac32}.$
From (\ref{t1}), we note that
\begin{equation*}
	\|n_{0}\|_{L^{\infty}L^{2}}\leq H_{2}.
\end{equation*}
Therefore $$\sup_{t\geq0}||n(t)||_{L^2}\leq|\mathbb{T}| ||n_{0}||_{L^\infty L^2}+||n_{\neq}||_{L^\infty L^2}\leq |\mathbb{T}|H_2+E_0,$$
and by interpolation, for $0<\theta<1$, we have
$$||n_{\rm in}||_{L^{2^j}}\leq||n_{\rm in}||^{\theta}_{L^2}
||n_{\rm in}||^{1-\theta}_{L^\infty}
\leq||n_{\rm in}||_{L^2}+||n_{\rm in}||_{L^\infty}\leq|\mathbb{T}|H_2+E_0+||n_{\rm in}||_{L^\infty},$$
for $j\geq1,$
which yields
$$2\int_{\mathbb{T}\times\mathbb{I}\times\mathbb{T}}|n(0)|^{2^{j+1}}dxdydz
\leq2\Big(|\mathbb{T}|H_2+E_0+||n_{\rm in}||_{L^\infty}\Big)^{2^{j+1}}\leq K^{2^{j+1}},$$
where $K=2(|\mathbb{T}|H_2+E_0+||n_{\rm in}||_{L^\infty}).$

We infer from (\ref{n1_2}) that
\begin{equation}
	\begin{aligned}
		\sup_{t\geq0}\int_{\mathbb{T}\times\mathbb{I}\times\mathbb{T}}|n(t)|^{2^{j+1}}dxdydz\leq \max\Big\{C_1	4096^{j}\Big(\sup_{t\geq0}\int_{\mathbb{T}\times\mathbb{I}\times\mathbb{T}}|n(t)|^{2^{j}}dxdydz\Big)^2, K^{2^{j+1}} \Big\}.\nonumber
	\end{aligned}
\end{equation}

When $j=1$, there holds
\begin{equation}
	\begin{aligned}
		\sup_{t\geq0}\int_{\mathbb{T}\times\mathbb{I}\times\mathbb{T}}|n(t)|^{2^{2}}dxdydz\leq C_1^{a_1}	4096^{b_1}K^{2^{2}},\nonumber
	\end{aligned}
\end{equation}
where $a_1=1$ and $b_1=1$.

When $j=2$, we have
\begin{equation}
	\begin{aligned}
		\sup_{t\geq0}\int_{\mathbb{T}\times\mathbb{I}\times\mathbb{T}}|n(t)|^{2^{3}}dxdydz\leq C_1^{a_2}	4096^{b_2}K^{2^{3}},\nonumber
	\end{aligned}
\end{equation}
where $a_2=1+2a_1$ and $b_2=2+2b_1$.

When $j=k$, we get
\begin{equation}
	\begin{aligned}
		\sup_{t\geq0}\int_{\mathbb{T}\times\mathbb{I}\times\mathbb{T}}|n(t)|^{2^{k+1}}dxdydz\leq C_1^{a_k}	4096^{b_k}K^{2^{k+1}},\nonumber
	\end{aligned}
\end{equation}
where $a_k=1+2a_{k-1}$ and $b_k=k+2b_{k-1}$.

Generally, one can obtain the following formulas
$$a_k=2^k-1,\ {\rm and}\ \ b_k=2^{k+1}-k-2.$$

Therefore,  one obtains
\begin{equation}
	\begin{aligned}
		\sup_{t\geq0}\left(\int_{\mathbb{T}\times\mathbb{I}\times\mathbb{T}}|n(t)|^{2^{k+1}}dxdydz\right)^{\frac{1}{2^{k+1}}}\leq C_1^{\frac{2^k-1}{2^{k+1}}}	4096^{\frac{2^{k+1}-k-2}{2^{k+1}}}K.\nonumber
	\end{aligned}
\end{equation}
Letting $k\rightarrow\infty$, there holds
$$\sup_{t\geq0}\|n(t)\|_{L^\infty}\leq C(E_{0}^{8}+H_{2}^{8})^{\frac{3}{4}}(|\mathbb{T}|H_2+E_0+||n_{\rm in}||_{L^\infty}):=E_{1}.$$

The proof is complete.
\end{proof}

\section{Proof of Theorem \ref{result 1}}\label{sec 7}
\begin{proof}[Proof of Theorem \ref{result 1}]
Recall (\ref{a7}) in the proof of {\textbf{Proposition \ref{priori0}}}, where we use the interpolation inequality
\begin{equation*}
	\|n_{(0,\neq)}\|_{L^{3}}\leq C_{*}\|n_{(0,\neq)}\|_{L^{1}}^{\frac13}\|\nabla n_{(0,\neq)}\|_{L^{2}}^{\frac23},
\end{equation*}
and  $ C_{*}^{3}M<1. $  From {\textbf{Lemma \ref{A1}}},  notice that the condition regarding the initial cell mass can be replaced by
$ \frac94M<1. $ Then combining it with {\textbf{Theorem \ref{result}}}, the proof is complete.

\end{proof}

\appendix
\section{}
We introduce the following time-space estimates with non-slip boundary condition and Navier-slip boundary condition, respectively, which play an important role in the estimates of the non-zero modes of the solution to the system (\ref{ini2})-(\ref{ini2_1}). 
\begin{proposition}[Proposition 10.1 in \cite{Chen1}]\label{slip}
Let $ f $ be a solution of 
\begin{equation*}
	\left\{
	\begin{array}{lr}
		\partial_tf-\frac{1}{A}\left(\partial_{y}^{2}-\eta^{2} \right)f+ik_{1}yf=-ik_{1}f_{1}-\partial_{y}f_{2}-ik_{3}f_{3}-f_{4}, \\
		f|_{y=\pm 1}=0,~~f|_{t=0}=f_{\rm in}, 
	\end{array}
	\right.
\end{equation*}
with $ f_{4}(t, \pm 1)=0 $ and $ f_{\rm in}(\pm 1)=0. $ Then there holds
\begin{equation*}
	\begin{aligned}
		&\|e^{aA^{-\frac13}t}f\|_{L^{\infty}L^{2}}^{2}+\frac{1}{A}\|e^{aA^{-\frac13}t}\partial_{y}f\|_{L^{2}L^{2}}^{2}+\big(A^{-1}\eta^{2}+(A^{-1}k_{1}^{2})^{\frac13} \big)\|e^{aA^{-\frac13}t}f\|_{L^{2}L^{2}}^{2}\\\leq&C\big(\|f_{\rm in}\|_{L^{2}}^{2}+A\|e^{aA^{-\frac13}t}f_{2}\|_{L^{2}L^{2}}^{2}+(\eta|k_{1}|)^{-1}\|e^{aA^{-\frac13}t}\partial_{y}f_{4}\|_{L^{2}L^{2}}^{2}+\eta|k_{1}|^{-1}\|e^{aA^{-\frac13}t}f_{4}\|_{L^{2}L^{2}}^{2}\\&+\min\{(A^{-1}\eta^{2})^{-1}, (A^{-1}k_{1}^{2})^{-\frac13} \}\|e^{aA^{-\frac13}t}(k_{1}f_{1}+k_{3}f_{3} )\|_{L^{2}L^{2}}^{2} \big),
	\end{aligned}
\end{equation*}	
where ``a'' is a non-negative constant, and $  f_{1}, f_{2}, f_{3}, f_{4} $ are given functions.	
\end{proposition}
\begin{proposition}[Proposition 10.2 in \cite{Chen1}]\label{noslip}
Let $f$	be a solution of
\begin{equation*}
	\left\{
	\begin{array}{lr}
		\partial_tf-\frac{1}{A}\left(\partial_{y}^{2}-\eta^{2} \right)f+ik_{1}yf=ik_{1}f_{1}+\partial_{y}f_{2}+ik_{3}f_{3}, \\
		\left(\partial_{y}^{2}-\eta^{2} \right)\varphi=f,\quad \partial_{y}\varphi|_{y=\pm 1}=\varphi|_{y=\pm 1}=0, \\
		f|_{t=0}=f_{\rm in}, 
	\end{array}
	\right.
\end{equation*}
with $\partial_{y}\varphi_{\rm in}|_{y=\pm 1}=0$. Then there holds
\begin{equation*}
	\begin{aligned}
		&|k_{1}\eta|^{\frac12}\|e^{aA^{-\frac13}t}(\partial_{y}, \eta)\varphi\|_{L^{2}L^{2}}+A^{-\frac34}\|e^{aA^{-\frac13}t}\partial_{y}f\|_{L^{2}L^{2}}+A^{-\frac12}\eta\|e^{aA^{-\frac13}t}f\|_{L^{2}L^{2}}\\&+\eta\|e^{aA^{-\frac13}t}(\partial_{y},\eta)\varphi\|_{L^{\infty}L^{2}}+A^{-\frac14}\|e^{aA^{-\frac13}t}f\|_{L^{\infty}L^{2}}\\\leq&CA^{\frac12}\|e^{aA^{-\frac13}t}(f_{1}, f_{2}, f_{3})\|_{L^{2}L^{2}}+C\left(\eta^{-1}\|\partial_{y}f_{\rm in}\|_{L^{2}}+\|f_{\rm in}\|_{L^{2}} \right),	
	\end{aligned}
\end{equation*}
where ``a'' is a non-negative constant, and $  f_{1}, f_{2}, f_{3}, \varphi $ are given functions.
\end{proposition}

The next lemma will be frequently used in the estimates of $ \|n_{(0,0)}\|_{L^{\infty}L^{2}}. $
\begin{lemma}\label{lemmaa1}
For given functions $f_1(y,z)$ and $f_2(y,z)$, satisfying $f_1(y,z), f_{2}(y,z)\in H^{1}\left(\mathbb{I}\times\mathbb{T} \right)$ , there holds
\begin{equation}\label{A.1}
	\|(f_1f_2)_{(0,0)}\|_{L^2}\leq C\|f_1\|_{L^2}\left(\|f_2\|^{\frac{1}{2}}_{L^2}
	\|\partial_yf_2\|_{L^2}^{\frac{1}{2}}+\|f_{2}\|_{L^{2}} \right).
\end{equation}
Especially, if $ f_{2}(y,z)\in H_{0}^{1}(\mathbb{I}\times\mathbb{T}) $, the result can be simplified as
\begin{equation}\label{A.2}
	\|(f_1f_2)_{(0,0)}\|_{L^2}\leq C\|f_1\|_{L^2}\|f_2\|^{\frac{1}{2}}_{L^2}
	\|\partial_yf_2\|_{L^2}^{\frac{1}{2}},
\end{equation}
where $ f_{i,(0,0)}=\frac{1}{|\mathbb{T}|}\int_{\mathbb{T}}f_{i}(y,z)dz, $ for $ i=1,2. $	
\end{lemma}
\begin{proof}
Let $f_1=\sum_{k_3\in\mathbb{Z}}f_{1}^{k_{3}}(y)e^{ik_3z}$	
and
$f_2=\sum_{k_3\in\mathbb{Z}}f_{2}^{k_{3}}(y)e^{ik_3z},$	
where
\begin{equation*}
	f_{1}^{k_{3}}(y)=\frac{1}{|\mathbb{T}|}\int_{\mathbb{T}}f_{1}(y,z)e^{-ik_{3}z}dz,\quad f_{2}^{k_{3}}(y)=\frac{1}{|\mathbb{T}|}\int_{\mathbb{T}}f_{2}(y,z)e^{-ik_{3}z}dz,
\end{equation*}
then 
$$(f_1f_2)_{(0,0)}=\sum_{k_3\in\mathbb{Z}}f_{1}^{k_{3}}(y)f_{2}^{-k_{3}}(y).$$
Notice that
\begin{equation}\label{A.3}
	\|f_{2}^{-k_{3}}(y)\|_{L^{\infty}}\leq C\left(\|f_{2}^{-k_{3}}(y)\|_{L^{2}}^{\frac12}\|\partial_{y}f_{2}^{-k_{3}}(y)\|_{L^{2}}^{\frac12}+\|f_{2}^{-k_{3}}(y)\|_{L^{2}} \right),
\end{equation}
then direct calculations yield that
\begin{equation*}
	\begin{aligned}
		&	\|(f_1f_2)_{(0,0)}\|_{L^2}\\=&\sum_{k_3\in\mathbb{Z}}\|f_{1}^{k_{3}}(y)f_{2}^{-k_{3}}(y)\|_{L^2}
		\leq \sum_{k_3\in\mathbb{Z}}\|f_{1}^{k_3}(y)\|_{L^2}\|f_{2}^{-k_3}(y)\|_{L^{\infty}}\\
		\leq& C\sum_{k_3\in\mathbb{Z}}\|f_{1}^{k_3}(y)\|_{L^2}\|f_{2}^{-k_3}(y)\|^{\frac{1}{2}}_{L^{2}}\|\partial_yf_{2}^{-k_3}(y)\|^{\frac{1}{2}}_{L^{2}}+C\sum_{k_{3}\in\mathbb{Z}}\|f_{1}^{k_{3}}(y)\|_{L^{2}}\|f_{2}^{-k_{3}}(y)\|_{L^{2}}\\\leq& C\left(\sum_{k_{3}\in\mathbb{Z}}\|f_{1}^{k_{3}}(y)\|_{L^{2}}^{2} \right)^{\frac12}\left(\sum_{k_{3}\in\mathbb{Z}}\|f_{2}^{k_{3}}(y)\|_{L^{2}}^{2} \right)^{\frac14}\left(\sum_{k_{3}\in\mathbb{Z}}\|\partial_{y}f_{2}^{k_{3}}(y)\|_{L^{2}}^{2} \right)^{\frac14}\\&+C\left(\sum_{k_{3}\in\mathbb{Z}}\|f_{1}^{k_{3}}(y)\|_{L^{2}}^{2} \right)^{\frac12}\left(\sum_{k_{3}\in\mathbb{Z}}\|f_{2}^{k_{3}}(y)\|_{L^{2}}^{2} \right)^{\frac12}\\\leq& C\|f_{1}\|_{L^{2}}\left(\|f_{2}\|_{L^{2}}^{\frac12}\|\partial_{y}f_{2}\|_{L^{2}}^{\frac12}+\|f_{2}\|_{L^{2}} \right),
	\end{aligned}
\end{equation*}
which implies (\ref{A.1}). 

Moreover, for $ f_{2}\in H_{0}^{1}(\mathbb{I}\times\mathbb{T}), $ there is no lower order term in (\ref{A.3}), thus (\ref{A.2}) holds. 
The proof is complete.
\end{proof}

To  estimate  $ \|n_{(0,\neq)}\|_{L^{\infty}L^{2}}, $ we need the following lemma to provide the bound-ness of $ \|\nabla c_{(0,\neq)}\|_{L^{\infty}}. $
\begin{lemma}\label{lemma_001}
Let $f(y,z)$ be a function such that $f_{(0,\neq)}\in H^2(\mathbb{I}\times \mathbb{T})$, then there holds	
$$||f_{(0,\neq)}||_{L^\infty}
\leq C\big\|\nabla f_{(0,\neq)}\big\|_{L^2}^{1-\epsilon}
\big\|
\partial_z\nabla f_{(0,\neq)}\big\|_{L^2}^{\epsilon},$$
where $\epsilon\in(0,1].$
\end{lemma}
\begin{proof}
For given 
$$f_{(0,\neq)}=\sum_{k_{3}\neq 0}f_{(0,\neq)}^{k_{3}}(y)e^{ik_{3}z},~~{\rm and}~~f_{(0,\neq)}^{k_{3}}(y)=\frac{1}{|\mathbb{T}|}\int_{\mathbb{T}}f_{(0,\neq)}(y,z)e^{-ik_{3}z}dz,$$
Gagliardo-Nirenberg inequality implies 
$$||f_{(0,\neq)}^{k_3}(t,y)||_{L^\infty_y}\leq C\left( 
||f_{(0,\neq)}^{k_3}(t,\cdot)||_{L^2}^{\frac{1}{2}}||\partial_y f_{(0,\neq)}^{k_3}(t,\cdot)||_{L^2}^{\frac{1}{2}}+||f_{(0,\neq)}^{k_3}(t,\cdot)||_{L^2}\right),$$
and we have
\begin{equation*}
	\begin{aligned}
		||f_{(0,\neq)}||_{L^\infty}\leq&\sum_{k_3\neq0}||f_{(0,\neq)}^{k_3}(t,y)||_{L^\infty_y}
		\\\leq& C\sum_{k_3\neq0}||f_{(0,\neq)}^{k_3}(t,\cdot)||_{L^2}^{\frac{1}{2}}||\partial_yf_{(0,\neq)}^{k_3}(t,\cdot)||_{L^2}^{\frac{1}{2}}+C\sum_{k_{3}\neq 0}||f_{(0,\neq)}^{k_3}(t,\cdot)||_{L^2}.
	\end{aligned}
\end{equation*}
Furthermore, by H\"{o}lder's inequality we get
\begin{equation}
	\begin{aligned}
		||f_{(0,\neq)}||_{L^\infty}
		\leq& C\Big(\sum_{k_{3}\neq0}|k_{3}|^{1+2\epsilon}||f_{(0,\neq)}^{k_3}(t,\cdot)||_{L^2}||\partial_yf_{(0,\neq)}^{k_3}(t,\cdot)||_{L^2}\Big)^{\frac{1}{2}}
		\Big(\sum_{k_{3}\neq0}\frac{1}{|k_{3}|^{1+2\epsilon}}\Big)^{\frac{1}{2}}\\&+C\left(\sum_{k_{3}\neq 0}|k_{3}|^{2\varepsilon} \|k_{3}f_{(0,\neq)}^{k_{3}}(t,\cdot)\|_{L^{2}}^{2} \right)^{\frac12}\left(\sum_{k_{3}\neq 0}\frac{1}{|k_{3}|^{1+\varepsilon}} \right) \\
		\leq& C \Big(\sum_{k_{3}\neq0}|k_{3}|^{2\epsilon}||k_{3}f_{(0,\neq)}^{k_3}(t,\cdot)||^2_{L^2}\Big)^{\frac{1}{4}}\Big(\sum_{k_{3}\neq0}|k_{3}|^{2\epsilon}\|  \partial_yf_{(0,\neq)}^{k_3}(t,\cdot)\|^2_{L^2}\Big)^{\frac{1}{4}}\\&+C\left(\sum_{k_{3}\neq 0}|k_{3}|^{2\varepsilon} \|k_{3}f_{(0,\neq)}^{k_{3}}(t,\cdot)\|_{L^{2}}^{2} \right)^{\frac12}, \nonumber
	\end{aligned}
\end{equation}
where $\epsilon\in(0,1].$	

Using  H\"{o}lder's inequality again, we have
\begin{equation}
	\begin{aligned}
		\sum_{k_{3}\neq0}|k_{3}|^{2\epsilon}||k_{3}f_{(0,\neq)}^{k_3}(t,\cdot)||^2_{L^2}
		&\leq C\sum_{k_{3}\neq0}\|k_{3}^2f_{(0,\neq)}^{k_3}(t,\cdot)\|^{2\epsilon}_{L^2}\|k_{3}f_{(0,\neq)}^{k_3}(t,\cdot)\|^{2(1-\epsilon)}_{L^2} \\
		&\leq C\big\|
		\partial_zf_{(0,\neq)}\big\|_{L^2}^{2(1-\epsilon)}
		\big\|
		\partial_{z}
		\nabla f_{(0,\neq)}\big\|_{L^2}^{2\epsilon}.
		\nonumber
	\end{aligned}
\end{equation}
Similarly, we have
\begin{equation}
	\begin{aligned}
		\sum_{k_{3}\neq0}|k_{3}|^{2\epsilon}||\partial_yf_{(0,\neq)}^{k_3}(t,\cdot)||^2_{L^2}
		&\leq C\|\partial_y f_{(0,\neq)}\|_{L^2}^{2(1-\epsilon)}
		\big\|
		\partial_{z}\nabla f_{(0,\neq)}\big\|_{L^2}^{2\epsilon}.
		\nonumber
	\end{aligned}
\end{equation}

Thus, we conclude that	
\begin{equation*}
	\begin{aligned}
		||f_{(0,\neq)}||_{L^\infty}\leq &C\big\|
		\partial_z f_{(0,\neq)}\big\|_{L^2}^{\frac{1-\epsilon}{2}}
		\|\partial_y f_{(0,\neq)}\|_{L^2}^{\frac{1-\epsilon}{2}}
		\big\|
		\partial_{z} \nabla f_{(0,\neq)}\big\|_{L^2}^{\epsilon}\\&+C\|\partial_{z}f_{(0,\neq)}\|_{L^{2}}^{1-\varepsilon}\|\partial_{z}\nabla f_{(0,\neq)}\|_{L^{2}}^{\varepsilon} 
		\\\leq& C\big\|\nabla f_{(0,\neq)}\big\|_{L^2}^{1-\epsilon}
		\big\|
		\partial_z\nabla f_{(0,\neq)}\big\|_{L^2}^{\epsilon},
	\end{aligned}
\end{equation*}
for $\epsilon\in(0,1].$ We complete the proof.
\end{proof}

The following lemma provides a specific embedding coefficient to determine an upper bound on the initial cell mass $ M. $
\begin{lemma}\label{A1}
Let $ f(y,z) $ be a function such that $ f_{(0,\neq)}\in H^{1}(\mathbb{I}\times\mathbb{T}) $ and $ f_{(0,\neq)}|_{y=\pm 1}=0, $ then there holds
\begin{equation}\label{eq A1}
	\|f_{(0,\neq)}\|_{L^{3}}^{3}\leq \frac94\|f_{(0,\neq)}\|_{L^{1}}\|\nabla f_{(0,\neq)}\|_{L^{2}}^{2}.
\end{equation}
\end{lemma}
\begin{proof}
Noting that
$ f_{(0,\neq)}|_{y=\pm 1}=0, $ then for any $ y\in[-1,1], $ we arrive at
\begin{equation*}
	f_{(0,\neq)}=\int_{-1}^{y}\partial_{y}f_{(0,\neq)}dy,
\end{equation*}
which indicates that
\begin{equation}\label{0}
	|f_{(0,\neq)}|\leq\int_{-1}^{1}|\partial_{y}f_{(0,\neq)}|dy.
\end{equation}	
Similarly
\begin{equation}\label{00}
	|f_{(0,\neq)}|\leq \int_{0}^{2\pi}|\partial_{z}f_{(0,\neq)}|dz.
\end{equation}
It follows from (\ref{0}) and (\ref{00}) that
\begin{equation}\label{000}
	\int_{\mathbb{I}\times\mathbb{T}}|f_{(0,\neq)}|^{2}dydz\leq \left(\int_{\mathbb{I}\times\mathbb{T}}|\nabla f_{(0,\neq)}|dydz \right)^{2}.
\end{equation}
Using (\ref{000}), there holds
\begin{equation*}
	\begin{aligned}
		\int_{\mathbb{I}\times\mathbb{T}}\left|f_{(0,\neq)}^{\frac32}\right|^{2}dydz\leq& \left(\int_{\mathbb{I}\times\mathbb{T}}\left|\nabla \left(f_{(0,\neq)}^{\frac32} \right)\right|dydz \right)^{2}\\
		\leq&\frac{9}{4}\|f_{(0,\neq)}\|_{L^{1}}\|\nabla f_{(0,\neq)}\|_{L^{2}}^{2},
	\end{aligned}
\end{equation*}
whh implies (\ref{eq A1}). The proof is complete.

\end{proof}

\section*{Acknowledgement}
The authors would like to thank Professors Zhifei Zhang, Zhaoyin Xiang and Ruizhao Zi for some helpful communications. W. Wang was supported by National Key R\&D Program of China (No. 2023YFA1009200), NSFC under grant 12071054, and by Dalian High-level Talent Innovation Project (Grant 2020RD09).
S. Cui's research was conducted while visiting McMaster University as a joint Ph.D. student. He expresses gratitude for China Scholarship Council's support and Professor Dmitry Pelinovsky's discussions.

\section*{Declaration of competing interest}
The authors declare that they have no known competing financial interests
or personal relationships that could have appeared to
influence the work reported in this paper.
\section*{Data availability}
No data was used in this paper.

\end{document}